\documentclass[1p]{elsarticle}

\makeatletter
\def\ps@pprintTitle{%
 \let\@oddhead\@empty
 \let\@evenhead\@empty
 \def\@oddfoot{\centerline{\thepage}}%
 \let\@evenfoot\@oddfoot}
\makeatother

\usepackage{lineno,hyperref}
\usepackage[utf8]{inputenc}
\usepackage{amsmath}
\usepackage{amsfonts}
\usepackage{amssymb}

\usepackage{amsthm}
\usepackage{mathtools}
\usepackage{bbm}
\usepackage{color}

\modulolinenumbers[5]

\theoremstyle{plain}
\newtheorem{theorem}{Theorem}[section]

\newtheorem{lemma}{Lemma}[section]
\newtheorem{proposition}{Proposition}[section]
\theoremstyle{definition}
\newtheorem{definition}{Definition}[section]
\theoremstyle{remark}

\numberwithin{equation}{section}

\biboptions{numbers}

\begin{document}

\begin{frontmatter}

\title{Vanishing Viscosity Limit of Short Wave-Long Wave Interactions in Planar Magnetohydrodynamics}


\author[mymainaddress]{Daniel R. Marroquin}
\ead{danielrm@impa.br}


\address[mymainaddress]{Instituto de Matematica Pura e Aplicada--IMPA, Estrada Dona Castorina, 110, Rio de Janeiro, RJ, 22460-320, Brazil}

\begin{abstract}
We study several mathematical aspects of a system of equations modelling the interaction between short waves, described by a nonlinear Schr\"{o}dinger equation, and long waves, described by the equations of magnetohydrodynamics for a compressible, heat conductive fluid. The system in question models an aurora-type phenomenon, where a short wave propagates along the streamlines of a magnetohydrodynamic medium. We focus on the one dimensional (planar) version of the model and address the problem of well posedness as well as convergence of the sequence of solutions as the bulk viscosity tends to zero together with some other interaction parameters, to a solution of the limit decoupled system involving the compressible Euler equations and a nonlinear Schr\"{o}dinger equation. The vanishing viscosity limit serves to justify the SW-LW interactions in the limit equations as, in this setting,  the SW-LW interactions cannot be defined in a straightforward way, due to the possible occurrence of vacuum.
\end{abstract}


\end{frontmatter}


\section{Introduction}

Motivated by the work of Benney in \cite{Be}, Dias and Frid \cite{DFr} proposed a model describing Short Wave-Long Wave interactions, where the short waves are given by a nonlinear Schr\"{o}dinger equation and the long waves are governed by the Navier–Stokes equations for a compressible isentropic fluid. The model describes the evolution of the wave function, obeying a nonlinear Schr\"{o}dinger equation, along the streamlines of the fluid flow. As such it can be stated through the following nonlinear Schr\"{o}dinger equation
\begin{equation}
i \psi_t + \psi_{yy} = |\psi|^2 \psi + G,\label{Sch1}
\end{equation}
where $i$ is the imaginary unit, $\psi=\psi(t,y)$ is the wave function, $G$ is a potential due to possible external forces and $y$ is the Lagrangian coordinate associated to the velocity field $u$ of the fluid. 

The Lagrangian coordinate is characterized by being constant along particle paths and can be defined accordingly through the relation
\begin{equation}
y(t,\Phi(t,x))=y_0(x),\label{Lag1}
\end{equation}
where $x$ is the Eulearian coordinate (which provides the usual spatial description of the dynamics of the fluid), $y_0$ is a suitable diffeomorphism which can be chosen conveniently and $\Phi=\Phi(t,x)$ is the flow of the fluid given by
\begin{equation}
\begin{cases}
\frac{d\Phi}{dt}(t;x)=u(t,\Phi(t;x)),\\
\Phi(0;x)=x.
\end{cases}\label{Lag2}
\end{equation}

In this paper we are interested in a similar model, where the Lagrangian coordinate is no longer given by the Navier-Stokes equations, but by the longitudinal velocity of a compressible, heat conductive magnetohydrodynamic (MHD) flow, being our main goal to investigate the convergence of the sequence of solutions as the bulk viscosity of the fluid tends to zero together with some other interaction parameters; focusing on the one (space) dimensional case of the equations.

This is a highly non-trivial problem since in the limit equations, when the bulk viscosity is zero, solutions are not smooth enough and vacuum is expected to occur in finite time, even if this is not the case in the initial data, which makes it impossible to properly define the Lagrangian transformation. Aside from this, to our knowledge, the vanishing viscosity problem has not yet been addressed in the case of a heat-conductive magnetohydrodynamic flow. It was solved by Chen and Perepelitsa \cite{CP} in the case of an isentropic flow, but they did not include neither the thermal description nor the magnetohydrodynamic coupling.

The phenomenon that we have in mind when we study this model is one like that of the auroras. Auroras, commonly known as polar lights, occur as fast-moving charged particles released from the sun collide with the Earth's atmosphere, channelled by Earth's magnetic field. The stream of charged particles, called solar wind, consists mainly of electrons, protons and alpha particles that, upon reaching the earth's magnetosphere, collide with atoms in the atmosphere, such as oxygen and nitrogen, imparting energy into them and thus making them excited. As the atoms return to their normal state they release photons, and when many of these collisions occur together they emit enough light for the phenomenon to be visible by the naked eye.

The aurora can thus be seen as small waves propagating along the trajectories of the particles of the atmosphere, a magnetohydrodynamic medium. 

Let us recall that the MHD equations describe the motion of a conductive fluid in the presence of a magnetic field. On the other hand, the nonlinear Schr\"{o}dinger equation describes collective phenomena in quantum plasmas. The example of the aurora gathers many of the ingredients captured by this model, and our results here provide insights about the behaviour of the solutions in a low viscosity regime. The limit process also serves the purpose of legitimizing the SW-LW interactions in the limit equations as, in this setting, the SW-LW interactions cannot be defined in a straightforward way due to the lack of regularity of the solutions, as well as the possible occurrence of vacuum.

Let us point out that Dias and Frid's model has been studied in the 3-dimensional context by Frid, Pan and Zhang \cite{FrPZ}, proving global existence and uniqueness of smooth solutions with small data. Later, Frid, Jia and Pan \cite{FrJP} extended these results to the 3-dimensional model involving the MHD equations, instead of the Navier-Stokes equations, showing decay rates on top of the global existence and uniqueness of smooth solutions also with small data.

In order to state precisely our results, we have to consider the system of equations that describe the dynamics of the fluid along which the wave function propagates. To that end, we consider the planar magnetohydrodynamics equations for a compressible fluid flow
\begin{flalign}
 &\rho_t+(\rho u)_x=0,&\label{MHDrho}\\
 &(\rho u)_t + \Big(\rho u^2 +  p +\frac{\beta}{2}|\mathbf{h}|^2\Big)_x = (\varepsilon u_x)_x +F_{\text{ext}},&\label{MHDu}\\
&(\rho \mathbf{w})_t + (\rho u\mathbf{w} - \beta\mathbf{h})_x =(\mu \mathbf{w}_x)_x,&\label{MHDw}\\
 &\mathcal{E}_t + \left( u\left(\mathcal{E} + p + \tfrac{\beta}{2}|\mathbf{h}|^2 \right) -\beta\mathbf{w}\cdot\mathbf{h}\right)_x &\nonumber  \\
  &\hspace{25mm}= (\kappa \theta_x)_x+(\varepsilon uu_x + \mu\mathbf{w}\cdot\mathbf{w}_x +\nu \mathbf{h}\cdot\mathbf{h}_x)_x , &\label{MHDE}\\
 &\beta\mathbf{h}_t+(\beta u\mathbf{h} - \beta\mathbf{w})_x=(\nu\mathbf{h}_x)_x.&\label{MHDh}
\end{flalign}

Here, $\rho\geq0$, $u\in\mathbb{R}$, $\mathbf{w}\in\mathbb{R}^2$ and $\theta\geq 0$ denote the fluid's density, longitudinal velocity, transverse velocity and temperature, respectively, $\mathbf{h}\in\mathbb{R}^2$ stands for the magnetic field; the total energy is
\[
\mathcal{E}:=\rho\left(e+\frac{1}{2}u^2+\frac{1}{2}|\mathbf{w}|^2\right)+\frac{\beta}{2}|\mathbf{h}|^2,
\]
with $e$ being the internal energy and $\frac{1}{2}|\mathbf{h}|^2$ the magnetic energy; $p$ denotes the pressure.

Furthermore, $\varepsilon>0$ is the bulk viscosity and $\mu>0$ the shear viscosity; $\kappa$ is the heat conductivity, $\nu>0$ is the magnetic difusivity and $\beta>0$ is the magnetic permeability.

The pressure and the internal energy, in general, depend on the density and the temperature through constitutive relations of the form
\[
p=p(\rho,\theta),\hspace{10mm}e=e(\rho,\theta),
\]
and must satisfy Maxwell's relation
\begin{equation}
e_\rho(\rho,\theta)=\frac{1}{\rho^2}(p-\theta p_\theta(\rho,\theta)).\label{MR}
\end{equation}

The planar MHD equations are deduced from the full three dimensional ones under the assumption that the flow moves in a preferential direction (the longitudinal direction) and is uniform in the transverse direction. This is translated into the equations by imposing that the partial derivatives with respect to the second and third spatial coordinates of the involved functions are equal to zero. Then, decomposing the velocity field as $(u,\mathbf{w})$ and the magnetic field as $(h_1,\mathbf{h})$, a straightforward calculation shows that $h_1$ is a constant, which can be assumed to be equal to $1$, and that the resulting equations are those above (see the appendix in \cite{CW}).

Coming back to the model under consideration, we consider the nonlinear Schr\"{o}dinger equation \eqref{Sch1} stated in the Lagrangian coordinate $y$ associated to the longitudinal velocity $u$, satisfying system \eqref{MHDrho}-\eqref{MHDh}. 

Note that, in this one dimensional setting, the Lagrangian coordinate $y=y(t,x)$ can be defined through the relations
\begin{equation}
y_x=\rho,\hspace{10mm}y_t=\rho u, \hspace{10mm}y(0,x)=\int_0^x \rho_0(z)dz,\label{Lag3}
\end{equation}
where $\rho_0(x)=\rho(0,x)$ is the initial density. Indeed, in light of equation \eqref{MHDrho}, $y$ is well defined and satisfies \eqref{Lag1} and \eqref{Lag2} with
\[
y_0(x)=\int_0^x \rho_0(z)dz.
\]

With this, the model is completed by choosing the external force $F_{\text{ext}}$ in \eqref{MHDu} and the potential $G$ in \eqref{Sch1} as
\begin{equation}
 F_{\text{ext}}=\alpha \Big(g'(1/\rho)h(|\psi\circ Y|^2)\Big)_x,\hspace{10mm}G=\alpha g(v)h'(|\psi|^2),\label{force}
\end{equation}
where $\alpha>0$ is the interaction coefficient, $Y(t,x)=(t,y(t,x))$ is the Lagrangian transformation, $v(t,y)$ is the \textit{specific volume} defined by
\begin{equation}
 v(t,y(t,x))=\frac{1}{\rho(t,x)},\label{SV}
\end{equation}
and $g,h:[0,\infty)\to[0,\infty)$ are nonnegative smooth functions.

Thus, we are left with the following system
\begin{flalign}
 &\rho_t+(\rho u)_x=0,&\label{E1rho}\\
 &(\rho u)_t + \Big(\rho u^2 +  p +\frac{\beta}{2}|\mathbf{h}|^2 - \alpha g'(1/\rho)h(|\psi\circ \mathbf{Y}|^2)\Big)_x = (\varepsilon u_x)_x ,&\label{E1u}\\
&(\rho \mathbf{w})_t + (\rho u\mathbf{w} - \beta\mathbf{h})_x =(\mu \mathbf{w}_x)_x,&\label{E1w}\\
 &\mathcal{E}_t + \left( u\left(\mathcal{E} + p + \tfrac{\beta}{2}|\mathbf{h}|^2 \right) -\beta\mathbf{w}\cdot\mathbf{h}\right)_x &\nonumber  \\
  &\hspace{5mm}= (\varepsilon uu_x + \mu\mathbf{w}\cdot\mathbf{w}_x +\nu \mathbf{h}\cdot\mathbf{h}_x+\kappa \theta_x)_x +\alpha \Big(g'(1/\rho)h(|\psi\circ \mathbf{Y}|^2) \Big)_x u, &\label{E1E}\\
 &\beta\mathbf{h}_t+(\beta u\mathbf{h} - \beta\mathbf{w})_x=(\nu\mathbf{h}_x)_x,&\label{E1h}\\
 &i\hspace{0.5mm}\psi_t+\psi_{yy}=|\psi|^2\psi+\alpha g(v)h'(|\psi|^2)\psi. & \label{E1Sch}
\end{flalign}

The most important feature of this coupling is that it is endowed with an energy identity which can be stated in differential form as
\begin{align}
 &\Bigg\{ \mathcal{E}_t + \left( u\left(\mathcal{E} + p + \tfrac{\beta}{2}|\mathbf{h}|^2 \right) -\beta\mathbf{w}\cdot\mathbf{h}\right)_x\nonumber \\
 &\hspace{20mm}-(\varepsilon uu_x + \mu\mathbf{w}\cdot\mathbf{w}_x +\nu \mathbf{h}\cdot\mathbf{h}_x+\kappa \theta_x)_x\Bigg\}dx\nonumber \\
 &\hspace{5mm}=\Bigg\{ (\overline{\psi}_t \psi_y + \psi_t\overline{\psi}_y)_y -  \Big(\tilde{\alpha}g(v) h(|\psi|^2) + \tfrac{1}{2} |\psi_y|^2 + \tfrac{1}{4}|\psi|^4\Big)_t \Bigg\}dy.\label{difE1}
\end{align}

Indeed, this identity can be easily deduced by multiplying \eqref{E1Sch} by $\overline{\psi}_t$ (the complex conjugate of $\psi_t$), taking real part, adding the resulting equation to the energy equation \eqref{E1E} and using relations \eqref{Lag3} in order to deal with the change of variables.

In particular, under suitable integrability conditions, this identity yields an integral form of the conservation of energy:
\begin{align*}
 &\frac{d}{dt}\int \mathcal{E}dx+\frac{d}{dt}\int \left( \frac{1}{2} |\psi_y(t,y)|^2  + \frac{1}{4}|\psi(t,y)|^4 + \alpha g(v(t,y)) h(|\psi(t,y)|^2) \right)dy=0.
\end{align*}

Now, for these calculations to hold and for the Lagrangian transformation $Y(t,\cdot)$ to actually be a change of variables, we need that the density $\rho(t,\cdot)$ be strictly positive for all $t\geq 0$. That is, we cannot admit the occurrence of vacuum.

Throughout this paper we are going to assume that the pressure can be decomposed into an elastic part, that depends only on the density, and a thermal part, that depends linearly on the temperature. Under this constraint we prove that the sequence of solutions to the system above converge to a weak solution of the limit system as the bulk viscosity tends to zero together with some other interaction parameters (specifically the thermal part of the pressure, the magnetic permeability and the interaction coefficient). For this, we develop some new uniform estimates that allow us to adapt and apply to our case the compactnes framework by Chen and Perepelitsa in \cite{CP} on the vanishing viscosity problem for the Navier Stokes equations. Through said estimates we are able to establish a certain rate at which these interactions parameters should tend to zero and with some careful analysis we manage to include the thermal description and the MHD coupling in the vanishing viscosity scheme.

As aforementioned, in the limit equations, the Lagrangian transformation cannot be properly defined. However, we can establish a relation between the limit velocity and the limit coordinate through the limit process so that the limit coordinate may be considered as a Lagrangian coordinate in a generalized sense.

Of course, before talking about the convergence of the sequence of solutions, we have to guarantee that system \eqref{E1rho}-\eqref{E1Sch} is well posed. In this direction, we are able to prove global existence and uniqueness of smooth solutions in a bounded open spacial domain $\Omega$. 

The rest of this paper is organized as follows. In Section \ref{results} we state precisely our results. In Section \ref{solutions} we prove the well posedness of system \eqref{E1rho}-\eqref{E1Sch}. Moving on to the vanishing viscosity problem, in section \ref{limit} we consider the limit equations and give an outline of the methods used in our compactness analysis. Section \ref{estimates} is devoted to the new uniform estimates that we develop, which allow us to adapt Chen and Perepelitsa's scheme. Finally in Section \ref{process} we explain the limit process.

\section{Main results}\label{results}

Throughout this work we assume that the pressure can be decomposed into an elastic part and a thermal part. More specifically, we consider a constitutive relation for the pressure of the form
\begin{equation}
p(\rho,\theta)=p_e(\rho)+  \theta p_\theta(\rho),\label{pe+ptet}
\end{equation}
where the elastic part $p_e$ is given by a $\gamma$-law:
\begin{equation}
p_e(\rho)=a\rho^\gamma,
\end{equation}
for some $a>0$ and $\gamma>1$. 

Concerning the thermal part of the pressure we assume that $p_\theta$ satisfies the following conditions:
\begin{equation}
\begin{cases}
p_\theta \in C[0,\infty)\cap C^1(0,\infty), &p_\theta(0)=0\\
p_\theta \text{ is a nondecreasing function of }\rho \in[0,\infty]\\
p_\theta(\rho)\leq p_0(1+\rho^\Gamma), &\text{for all }\rho\geq 0,
\end{cases}\label{ptet}
\end{equation}
for some $p_0 \geq 0$ and $\Gamma \leq \frac{\gamma}{2}$.

This choice of constitutive relation agrees with the one considered by Feireisl in \cite{Fe}, where he proves existence of weak solutions to the full multidimensional Navier-Stokes equations, and we refer to it for a wide discussion on its physical relevance.

Now, according to Maxwell's relation \eqref{MR} the internal energy can be written in the form
\begin{equation}
e(\rho,\theta)=P_e(\rho) + Q(\theta),\label{defe}
\end{equation}
with $P_e$ given by
\begin{equation}
P_e(\rho)=\frac{a}{\gamma-1}\rho^{\gamma-1},\label{elpot}
\end{equation}
and $Q(\theta)$ given by
\begin{equation}
Q(\theta)=\int_0^\theta C_\vartheta(z)dz,
\end{equation}
where $C_\vartheta(\theta):=\partial e/\partial \theta$ is the specific heat at constant volume, which depends only on the temperature.

Regarding the function $C_\vartheta$ we assume that:
\begin{equation}
\begin{cases}
C_\vartheta \in C^1[0,\infty), &\inf_{z\in [0,\infty)}C_\vartheta (z)>0 \\
e_1(1+\theta^r)\leq C_\vartheta(\theta)\leq e_2 (1+\theta^r), &
\end{cases}\label{Q}
\end{equation}
where $r\in[0,1]$ and $e_1$ and $e_2$ are appropriate positive constants.

Let us point out that, in view of \eqref{pe+ptet} and \eqref{defe}, the energy equation \eqref{E1E} is equivalent to
\begin{equation}
(\rho Q(\theta))_t + (\rho Q(\theta)\mathbf{u})_x+\theta p_\theta(\rho)u_x =(\kappa \theta_x)_x+ \varepsilon u_x^2 + \mu |\mathbf{w}_x|^2 + \nu|\mathbf{h}_x|^2.\label{E1Q}
\end{equation}

For later reference, let us also introduce the specific entropy $s=s(\rho,\theta)$ through the thermodynamic relations
\begin{equation}
\theta s_\rho = e_\rho - \frac{p}{\rho^2},\hspace{10mm}\theta s_\theta=e_\theta,\label{entropydef}
\end{equation}
that is
\begin{equation}
s(\rho,\theta)=\int_1^\theta \frac{C_\vartheta(z)}{z}dz-\int_1^\rho \frac{p_\theta(z)}{z^2}dz.\label{entropy}
\end{equation}

In addition to this, as in \cite{Fe,CW,W}, we have to impose that the heat conductivity coefficient $\kappa$ depend on the temperature and obey some growth conditions for our results to hold. Specifically, we assume that $\kappa=\kappa(\theta)$ satisfies:
\begin{equation}
\begin{cases}
\kappa \in C^2([0,\infty)) &\\
k_1(1+\theta^q)\leq \kappa(\theta)\leq k_2 (1+\theta^q), &\text{for all }\theta\geq 0\\
\kappa_\theta(\theta)\leq k_2 (1+\theta^{q'}), & \text{for all } \theta\geq 0.
\end{cases}\label{K}
\end{equation} 
Here, $k_1>0$, $q \geq 2+2r$, $q'\geq 0$ and $r$ is the same as in (\ref{Q}).

Finally we impose some conditions on the functions involved in the coupling describing the short wave-long wave interaction
\begin{equation}
\begin{cases}
g,h:[0,\infty)\to [0,\infty), \text{ smooth with } g(0)=h(0)=0,\\
 \text{supp}\hspace{.5mm} g' \text{ compact in }(0,\infty),\\
 \text{supp}\hspace{.5mm} h' \text{ compact in } [0,\infty).
\end{cases}\label{gh}
\end{equation} 

Under these conditions we consider the initial-boundary value problem for the system (\ref{E1rho})-(\ref{E1Sch}) in a bounded spatial domain $\Omega$, which we can assume to be $(0,1)$ without loss of generality, with the following initial and boundary conditions
\begin{equation}
\begin{cases}
(\rho,u,\mathbf{w},\mathbf{h},\theta)|_{t=0}=(\rho_0,u_0,\mathbf{w}_0,\mathbf{h}_0,\theta_0)(x), &x\in \Omega,\\
\psi|_{t=0}=\psi_0(y), &y\in \Omega_y,\\
(u,\mathbf{w},\mathbf{h},\theta_x)|_{\partial\Omega}=0, \hspace{10mm}\psi|_{\partial \Omega_y}=0, &t>0,
\end{cases}\label{E10}
\end{equation}
where the initial data satisfies the respective compatibility conditions. Here $\Omega_y$ is the domain of the Lagrangian coordinate.

Then, we have the following result.

\begin{theorem}\label{thm2}
Supose that there are positive constants $m<M$ such that
\begin{align}
&m \leq \rho_0(x),\theta_0(x) \leq M, &x\in\Omega,\label{cotas0}
\end{align}
and that
\begin{equation}
\rho_0,u_0,\mathbf{w}_0,\mathbf{h}_0,\theta_0\in H^1(\Omega),\hspace{5mm}\psi_0\in H^2(\Omega_y;\mathbb{C}),\label{reg0}
\end{equation}
and $\rho_0\in W^{1,\infty}(\Omega)$. Then, problem (\ref{E1rho})-(\ref{E1Sch}), (\ref{E10}) has a unique global solution $(\rho,u,\mathbf{w},\mathbf{h},\theta)(t,x)$, $\psi(t,y)$ such that for any fixed $T>0$
\begin{align*}
&\rho\in L^\infty(0,T;H^1(\Omega)\cap W^{1,\infty}(\Omega)),\\
&(u,\mathbf{w},\mathbf{h})\in L^\infty(0,T;H_0^1(\Omega))\cap L^2(0,T;H^2(\Omega)),\\
&\theta \in L^\infty(0,T;H^1(\Omega)),\hspace{5mm}\theta_y\in L^2(0,T;H_0^1(\Omega))\\
&\psi\in L^\infty(0,T;H_0^1(\Omega_y;\mathbb{C})\cap H^2(\Omega_y;\mathbb{C})).
\end{align*}
Also, for each $(t,x)\in [0,T]\times\Omega$ we have
\[
C^{-1}\leq \rho(t,x), \theta(t,x)\leq C,,
\]
where $C>0$ is a constant depending only on $T,m,M$ and the initial data.
\end{theorem}

Having well posedness for system \eqref{E1rho}-\eqref{E1Sch} we can move on to the vanishing viscosity problem. To that end, we introduce a new artificial small parameter $\delta$ multiplying the thermal part of the pressure. That is, we substitute relation \eqref{pe+ptet} by  
\begin{equation}
p(\rho,\theta)=a\rho^\gamma + \delta\theta p_\theta(\rho),\label{deltape+ptet}
\end{equation}
where $\delta$ is some positive constant. This certainly agrees with our previous assumptions and Theorem \ref{thm2} continues to hold.

Note, however, that if we take $\varepsilon=\alpha=\beta=\delta=0$ then we are left with a decoupled system involving the compressible one dimensional Euler Equations and the nonlinear Schr\"{o}dinger equation. Namely,
\begin{flalign}
 &\rho_t+(\rho u)_x=0,&\label{E1rhoinfty}\\
 &(\rho u)_t + (\rho u^2 +  a\rho^\gamma )_x=0 ,&\label{E1uinfty}\\
&(\rho \mathbf{w})_t + (\rho u\mathbf{w} )_x =(\mu \mathbf{w}_x)_x,&\label{E1winfty}\\
&(\rho Q(\theta))_t + (\rho Q(\theta)\mathbf{u})_x=(\kappa \theta_x)_x + \mu |\mathbf{w}_x|^2 + \nu|\mathbf{h}_x|^2, &\label{E1Qinfty}\\
 &(\nu\mathbf{h}_x)_x=0,&\label{E1hinfty}\\
 &i\hspace{0.5mm}\psi_t+\psi_{yy}=|\psi|^2\psi. & \label{E1Schinfty}
\end{flalign}

Our next task is to study this system and its relation with our original viscous system \eqref{E1rho}-\eqref{E1Sch}. More precisely, based on some new uniform estimates, we show that the sequence of solutions to the viscous system, given by Theorem \ref{thm2}, converges to a weak solution of the limit problem above as $(\varepsilon,\alpha,\beta,\delta)\to 0$. 

Said estimates pose, as will be shown later, some restriction in the way that these coefficients vanish. Namely, $\alpha = o(\varepsilon^{1/2})$, $\beta = o(\varepsilon)$ and $\delta = o(\varepsilon)$ as $\varepsilon \to 0$. As such, we can, for simplicity, consider $\alpha$, $\beta$ and $\delta$ as functions of $\varepsilon$ and consider a sequence of solutions $(\rho^\varepsilon,u^\varepsilon,\mathbf{w}^\varepsilon,\mathbf{h}^\varepsilon,\theta^\varepsilon,\psi^\varepsilon)$ to \eqref{E1rho}-\eqref{E1Sch} with initial data $(\rho_0^\varepsilon,u_0^\varepsilon,\mathbf{w}_0^\varepsilon,\mathbf{h}_0^\varepsilon,\theta_0^\varepsilon,\psi_0^\varepsilon)$. With this notation, our main result reads as follows.

\begin{theorem}\label{vanish}
Let the initial functions $(\rho_0^\varepsilon,u_0^\varepsilon,\mathbf{w}_0^\varepsilon,\mathbf{h}_0^\varepsilon,\theta_0^\varepsilon,\psi_0^\varepsilon)$ be smooth and satisfy the following conditions:
\begin{itemize}
\item[(i)] $\rho_0^\varepsilon\geq c_0^\varepsilon$, $M_0^{-1}\leq\int_\Omega \rho_0^\varepsilon dx \leq M_0$, $\int_\Omega \rho_0^\varepsilon |u_0^\varepsilon|^2 dx\leq M_0$, $-\int_\Omega\rho_0^\varepsilon s(\rho_0^\varepsilon,\theta_0^\varepsilon)dx\leq M_0$, for some $M_0$ independent of $\varepsilon$ and some $c_0^\varepsilon>0$;
\item[(ii)] $\int_\Omega (\rho_0^\varepsilon(|u_0^\varepsilon|^2+|\mathbf{w}_0^\varepsilon|^2)+|\beta\mathbf{h}_0^\varepsilon|^2)dx + \int_{\Omega_y}(|\psi_{0y}^\varepsilon|^2+|\psi_0^\varepsilon|^2 )dy\leq M_0$;
\item[(iii)] $\varepsilon^2\int_\Omega |\rho_{0x}^\varepsilon|^2|\rho_{0}^\varepsilon|^{-3}dx+\varepsilon\beta^2 \int_\Omega |\mathbf{h}_0^\varepsilon|^2(\rho_0^\varepsilon)^{-1}dx \leq M_0$;
\item[(iv)] $(\rho_0^\varepsilon,\rho_0^\varepsilon u_0^\varepsilon)\to (\rho_0,\rho_0 u_0)$ as $\varepsilon\to 0$ in the sense of distributions, with $\rho_0\geq 0$ a.e. 
\item[(v)] $\rho_0^\varepsilon\mathbf{w}_0^\varepsilon\to \rho_0\mathbf{w}_0$ and $\beta\mathbf{h}_0^\varepsilon\to 0$ in the sense of distributions;
\item[(vi)] $\rho_0^\varepsilon Q(\theta_0^\varepsilon)\to \rho_0 Q(\theta_0)$  in the sense of distributions.
\item[(vii)] $\psi_0^\varepsilon\to\psi_0$ in $H_0^1(\Omega_y)$.
\end{itemize}

Let $(\rho^\varepsilon,u^\varepsilon,\mathbf{w}^\varepsilon,\mathbf{h}^\varepsilon,\theta^\varepsilon,\psi^\varepsilon)$ be the solution of \eqref{E1rho}-\eqref{E1Sch} with $p$ given by \eqref{deltape+ptet} and with initial data $(\rho_0^\varepsilon,u_0^\varepsilon,\mathbf{w}_0^\varepsilon,\mathbf{h}_0^\varepsilon,\theta_0^\varepsilon,\psi_0^\varepsilon)$. Assume, further that $\alpha=o(\varepsilon^{1/2})$, $\beta=o(\varepsilon)$ and $\delta=o(\varepsilon)$. Then, we may extract a subsequence (not relabelled) of $(\rho^\varepsilon,u^\varepsilon,\mathbf{w}^\varepsilon,\mathbf{h}^\varepsilon,\theta^\varepsilon,\psi^\varepsilon)$ such that as $\varepsilon\to 0$ we have
\begin{itemize}
\item $(\rho^\varepsilon,\rho^\varepsilon u^\varepsilon)$ converges in $L_{loc}^1(\Omega\times (0,\infty))$ to a \textbf{finite-energy entropy solution} $(\rho,\rho u)$ of the compressible Euler equations \eqref{E1rhoinfty}, \eqref{E1uinfty} with initial data $(\rho_0,\rho_0 u_0)$;
\item $(\mathbf{w}^\varepsilon,\mathbf{h}^\varepsilon)\to (\mathbf{w},0)$ weakly in $L^2(0,T;H_0^1(\Omega))$ and $(\rho,\rho u,\mathbf{w})$ solve equation \eqref{E1winfty} in the sense of distributions with initial data $\rho_0 \mathbf{w}_0$ attained also in the sense of distributions;
\item $\psi^\varepsilon\to\psi$ strongly in $L^\infty(0,T;L^4(\Omega))$ and weakly-* in $L^\infty(0,T;H_0^1(\Omega))$, where $\psi$ is the unique \textbf{weak solution} of equation \eqref{E1Schinfty} with initial data $\psi_0$;
\item $\rho^\varepsilon Q(\theta^\varepsilon)$ converges strongly to $\rho Q(\theta)$ in $L_{loc}^1(\Omega\times(0,\infty))$ and $(\rho, \rho u,\mathbf{w}, \theta)$ constitute a \textbf{variational solution} of equation \eqref{E1Qinfty}.
\end{itemize}
\end{theorem}

For simplicity we state the definitions of finite-energy entropy solution to the Euler Equations, weak solution to the nonlinear Schr\"{o}dinger equation and variational solution of the thermal energy equation only in Section \ref{limit} where we discuss some generalities about the limit equations \eqref{E1rhoinfty}-\eqref{E1Schinfty}.

It is worth mentioning that the magnetic permeability $\beta$, which relates the magnetic field to the magnetic induction, is usually taken to be equal to $1$ in the literature (\cite{LL}) since in most real world media covered by the model this constant differs only slightly from the unity. However, the only physical restriction on it is its positivity.

As aforementioned, we are inspired by the work of Dias and Frid in \cite{DFr} who pursue similar objectives on a SW-LW interactions model involving the isentropic Navier-Stokes equations and who, in turn, follow the work by Chen and Perepelitsa in \cite{CP} on the vanishing viscosity limit for the isentropic one dimensional Navier-Stokes equations. Our main contribution here is to include the thermal description as well as the electromagnetic coupling.

\section{Existence and uniqueness of solutions}\label{solutions}

This Section is devoted to the proof of Theorem \ref{thm2}. To that end we write the whole system in Lagrangian coordinates and prove well posedness for the resulting system. In particular, we show that no vacuum nor concentration develops in finite time. This implies that the Lagrangian transformation is smooth and invertible and therefore we can turn back to the original Eulerean coordinates to conclude.

In order to prove well posedness for the system in the Lagrangian variables we first prove existence and uniqueness of local solutions and then extend the local solutions to global ones based on a priori estimates.

For the local result use a Faedo-Galerkin type method similar to the one applied by Dias and Frid in \cite{DFr}, which in turn resembles the classic work by Kazhikhov and Shelukhin in \cite{KzSh} (cf. \cite[Chapter~2]{AKM}). As for the global result, we develop some a priori estimates inspired by the work of Chen and Wang in \cite{CW} and by the work of Wang in \cite{W}. 

\subsection{Lagrangian coordinates}

Using relations \eqref{Lag3}, a straightforward calculation shows that system \eqref{E1rho}-\eqref{E1Sch} is equivalent to
\begin{flalign}
&v_t - u_y=0,&\label{L1v}\\
&u_t + \left(p+\frac{\beta}{2}|\mathbf{h}|^2 - \alpha g'(v)h(|\psi|^2)\right)_y=\left(\frac{\varepsilon  u_y}{v}\right)_y, &\label{L1u}\\
&\mathbf{w}_t - \beta\mathbf{h}_y=\Big(\frac{\mu \mathbf{w}_y}{v}\Big)_y, &\label{L1w}\\
&\Big[ e+\tfrac{1}{2}(u^2+|\mathbf{w}|^2+\beta v|\mathbf{h}|^2) + \alpha g(v)h(|\psi|^2) + \tfrac{1}{2}|\psi_y|^2 + \tfrac{1}{2}|\psi|^4  \Big]_t &\nonumber\\
 &\hspace{10mm} + \Big( u\Big(p+\frac{\beta}{2}|\mathbf{h}|^2 - \alpha g'(v)h(|\psi|^2)\Big) - \beta \mathbf{h}\cdot\mathbf{w} - (\psi_t\overline{\psi}_y + \overline{\psi}_t\psi_y)\Big)_y &\nonumber\\
 &\hspace{20mm}=\Big( \frac{\kappa \theta_y}{v} + \frac{\varepsilon uu_y}{v} + \frac{\mu \mathbf{w}\cdot\mathbf{w}_y}{v} + \frac{\nu \mathbf{h}\cdot\mathbf{h}_y}{v} \Big)_y, &\label{L1E}\\
&(\beta v\mathbf{h})_t - \beta \mathbf{w}_y=\Big(\frac{\nu \mathbf{h}_y}{v}\Big)_y, &\label{L1h}\\
&i\hspace{0.5mm}\psi_t+\psi_{yy}=|\psi|^2\psi+\tilde{\alpha}g(v)h'(|\psi|^2)\psi. &\label{L1Sch}
\end{flalign}
where, $v$ is the specific volume given by (\ref{SV}). Accordingly, equation \eqref{E1Q} results in
\begin{equation}
Q(\theta)_t + \theta p_\theta(\rho)u_y =(\frac{\kappa \theta_y}{v})_y+ \frac{\varepsilon u_y^2}{v} + \frac{\mu |\mathbf{w}_y|^2}{v} + \frac{\nu|\mathbf{h}_y|^2}{v}.\label{L1Q}
\end{equation}

Of course, this change of variables is justified only when $\rho$ is finite and strictly positive. Note that equation \eqref{E1rho} together with the boundary conditions \eqref{E10} on $u$ imply that
\[
\int_\Omega \rho(t,x)dx = \int_\Omega \rho_0(x)dx,
\]
and in view of \eqref{Lag3} we have that $\Omega_y=(0,d)$, where $d$ is the value of the integral above. For simplicity, we assume without loss of generality that $d=1$ and to avoid the overload of notation, in this section we omit the subindex of the domain of the Lagrangian coordinate and write it simply as $\Omega$.

Now, let us consider the initial-boundary value problem for system \eqref{L1v}-\eqref{L1Sch} on $\Omega$ with the following initial and boundary conditions
\begin{equation}
\begin{cases}
(v,u,\mathbf{w},\mathbf{h},\theta,\psi)|_{t=0}=(v_0,u_0,\mathbf{w}_0,\mathbf{h}_0,\theta_0,\psi_0)(y), &y\in\Omega,\\
(u,\mathbf{w},\mathbf{h},\theta_y,\psi)|_{\partial\Omega}=0.
\end{cases}\label{conds}
\end{equation}

In connection with \eqref{cotas0} and \eqref{reg0} we assume that
\begin{align}
&m \leq v_0(y),\theta_0(y) \leq M, &y\in\Omega\label{cotasL0}
\end{align}
and that
\begin{equation}
v_0\in H^1(\Omega)\cap W^{1,\infty}(\Omega),\hspace{5mm} u_0,\mathbf{w}_0,\mathbf{h}_0,\theta_0\in H^1(\Omega),\hspace{5mm}\psi_0\in H^2(\Omega;\mathbb{C}),\label{regL0}
\end{equation}

Remember that we made some assumptions on the pressure, internal energy, heat conductivity and coupling functions. In connection with \eqref{pe+ptet} and \eqref{defe}, by an abuse of notation, we have that $p=p(v,\theta)$ is given by
\begin{equation}
p(v,\theta)=p_e(v) + \theta p_\theta(v),\label{pe+ptetv}
\end{equation}
where the elastic part $p_e$ is given by
\begin{equation}
p_e(\rho)=a v^{-\gamma},
\end{equation}
with $a>0$ and $\gamma >1$. Concerning the thermal part of the pressure $p_\theta$ we assume that
\begin{equation}
\begin{cases}
p_\theta \in C(0,\infty)\cap C^1(0,\infty), &\lim_{v\to \infty}p_\theta(v)=0\\
p_\theta \text{ is a nonincreasing function of }v \in (0,\infty)\\
p_\theta(v)\leq p_0(1+v^{-\Gamma}), &\text{for all }\rho\geq 0,
\end{cases}\label{ptetv}
\end{equation}
for some $p_0 \geq 0$ and $\Gamma \leq \frac{\gamma}{2}$.

Accordingly, the internal energy $e=e(v,\theta)$ is given by
\begin{equation}
e(v,\theta)=P_e(v) + Q(\theta),\label{ePe+Qv}
\end{equation}
where
\begin{align}
&P_e(v)=\frac{a}{\gamma-1}v^{1-\gamma}, &Q(\theta)=\int_0^\theta C_\vartheta (z)dz.\label{PeQv}
\end{align}
Concerning the function $C_\vartheta$ we assume \eqref{Q}.

As before, the heat conductivity $\kappa$ must depend on $\theta$ and satisfy \eqref{K}. Moreover, we assume that the coupling functions $g$ and $h$ satisfy \eqref{gh}. Finally, we assume that the parameters $\varepsilon, \mu, \nu, \beta$ and $\alpha$ are fixed positive constants. 

\subsection{Local solutions}

First, we prove the existence of local solutions.

\begin{lemma}\label{local}
Let us assume that the initial data $(v_0,u_0,\mathbf{w}_0,\mathbf{h}_0,\theta_0,\psi_0)(y)$ satisfies
\begin{align}
&m < v_0(y),\theta_0(y) < M, &y\in\Omega,\label{m.v.M}
\end{align}
and that
\begin{equation}
v_0,u_0,\mathbf{w}_0,\mathbf{h}_0,\theta_0\in H^1(\Omega),\hspace{5mm}\psi_0\in H^1(\Omega;\mathbb{C}).\label{reg0.}
\end{equation}

Then, there exist $T>0$ and a solution of (\ref{L1v})-(\ref{L1Sch}), (\ref{conds}) satisfying
\begin{align*}
&v\in C([0,T];H^1(\Omega)),\hspace{5mm} \frac{m}{4}\leq v \leq 4M\\
&(u,\mathbf{w},\mathbf{h})\in C([0,T];H_0^1(\Omega))\cap L^2(0,T;H^2(\Omega)),\\
&\theta \in C([0,T];H^1(\Omega)),\hspace{5mm}\theta_y\in L^2(0,T;H_0^1(\Omega)), \hspace{5mm} \theta>0\\
&\psi\in C([0,T];H_0^1(\Omega;\mathbb{C})),\\
&v_t,u_t,\mathbf{w}_t,\mathbf{h}_t,\theta_t\in L^2(0,T;L^2(\Omega)).
\end{align*}
\end{lemma}

\begin{proof}
Let us construct a sequence of approximate solutions $(v^n,u^n,\mathbf{w}^n,\mathbf{h}^n,\theta^n,$ $\psi_n)$ where $(u^n,\mathbf{w}^n,\mathbf{h}^n,\theta^n,\psi^n)$ are of the form
\begin{align}
&u^n(t,y)=\sum_{k=1}^n u_k^n(t)sin(k\pi y),&\mathbf{w}^n(t,y)=\sum_{k=1}^n \mathbf{w}_k^n(t)sin(k\pi y),\nonumber\\
&\mathbf{h}^n(t,y)=\sum_{k=1}^n \mathbf{h}_k^n(t)sin(k\pi y),&\theta^n(t,y)=\sum_{j=0}^n \theta_j^n(t)cos(j\pi y),\label{appsols}\\
&\psi^n(t,y)=\sum_{k=1}^n \psi_k^n(t)sin(k\pi y),&\hspace{20mm}n=1,2,...\nonumber
\end{align}

Note that each approximation is written as a sum of either sines or cosines so that they match the desired boundary conditions (for example, $\theta^n_y|_{\partial \Omega}=0$).

In order to determine the coefficients $u_k^n(t), \mathbf{w}_k^n(t), \mathbf{h}_k^n(t), \theta_j^n(t), \psi_k^n(t)$, $j=0,1,...,n$, $k=1,...,n$, we demand that equations (\ref{L1u})-(\ref{L1Sch}) be satisfied in an approximate way. To this end, we consider the spaces
\begin{align*}
&\mathcal{S}_n:=\text{span}_\mathbb{C}\{ sin(k\pi y) : k=1,...,n \},&\mathcal{C}_n:=\text{span}_\mathbb{C}\{ cos(j\pi y) : j=0,1,...,n \},
\end{align*}
with respective projections $P_n^{\mathcal{S}}:L^2(\Omega)\to \mathcal{S}_n$ and $P_n^{\mathcal{C}}:L^2(\Omega)\to \mathcal{C}_n$.

By virtue of (\ref{L1v}) we take 
\begin{equation}
v^n(t,y):=v_0(y)+\int_0^t u_y^n(y,s)ds,\label{Lnv}
\end{equation}
so that,
\[
v_t^n=u_y^n,\hspace{10mm} v^n|_{t=0}=v_0.
\]

With the notation above, we consider the following system.

\begin{flalign}
&u_t^n = P_n^{\mathcal{S}}\Bigg[\left(-p(v^n,\theta^n)-\frac{\beta}{2}|\mathbf{h}^n|^2 + \alpha g'(v^n)h(|\psi^n|^2)\frac{\varepsilon  u_y^n}{v^n}\right)_y\Bigg],&\label{Lnu}\\
&\mathbf{w}_t^n =P_n^{\mathcal{S}}\Bigg[\beta\mathbf{h}_y^n + \left(\frac{\mu \mathbf{w}_y^n}{v^n}\right)_y\Bigg],&\label{Lnw}\\
&\beta \mathbf{h}^n_t= P_n^{\mathcal{S}}\Bigg[ \frac{1}{v^n}\Bigg( -\beta u_y^n\mathbf{h}^n + \beta \mathbf{w}^n_y + \left(\frac{\nu \mathbf{h}_y^n}{v^n}\right)_y \Bigg)\Bigg],&\label{Lnh}\\
&\theta_t^n= P_n^{\mathcal{C}}\Bigg[ \frac{1}{C_\vartheta(\theta^n)}\Bigg(-\theta^n p_\theta(v^n)u_y^n + \left( \frac{\kappa(\theta^n)\theta_y^n}{v^n} \right)_y & \nonumber\\
  &\hspace{50mm}+\frac{\varepsilon |u_y^n|^2}{v^n} + \frac{\mu |\mathbf{w}_y^n|^2}{v^n} + \frac{\nu |\mathbf{h}_y^n|^2}{v^n} \Bigg) \Bigg],&\label{LnE}\\
&i\hspace{0.5mm}\psi_t^n = P_n^{\mathcal{S}}\Bigg[ -\psi_{yy}^n + |\psi^n|^2\psi^n + \alpha g(v^n)h'(|\psi^n|^2)\psi^n \Bigg].&\label{LnSch}
\end{flalign}

Now, system (\ref{Lnu})-(\ref{LnSch}) poses a system of ODE's for the coefficients $u_k^n(t)$, $\mathbf{w}_k^n(t)$, $\mathbf{h}_k^n$, $\theta_j^n(t)$, $\psi_k^n(t)$, $k=1,2,...,n$, $j=0,1,...,n$. 

Regarding the initial conditions, we impose that 
\begin{equation}
(u^n,\mathbf{w}^n,\mathbf{h}^n,\theta^n,\psi^n)|_{t=0}=(u_0^n,\mathbf{w}_0^n,\mathbf{h}_0^n,\theta_0^n,\psi_0^n),\label{Ln0}
\end{equation}
where the latter satisfy
\begin{equation}
u_0^n,\mathbf{w}_0^n,\mathbf{h}_0^n,\psi_0^n \in \mathcal{S}_n,\hspace{10mm}\theta_0^n\in\mathcal{C}_n,
\end{equation}
and
\begin{equation}
(u_0^n,\mathbf{w}_0^n,\mathbf{h}_0^n,\theta_0^n,\psi_0^n)\to (u_0,\mathbf{w}_0,\mathbf{h}_0,\theta_0,\psi_0)
\end{equation}
in $H^1(\Omega)$ (and, therefore, uniformly).

Taking the coefficients of the newly defined approximate initial data as initial conditions for the respective coefficients and taking into account relation (\ref{Lnv}), the existence and uniqueness of solutions of (\ref{Lnu})-(\ref{LnSch}) are guaranteed by the well known classical results on the theory of ordinary differential equations.

Having a sequence of approximate solutions we now need some uniform estimates that allow us to take a convergent subsequence to a solution of the original problem (\ref{L1v})-(\ref{L1Sch}), (\ref{conds}).

Observe that each one of the approximate solutions $(v^n,u^n,\mathbf{w}^n,\mathbf{h}^n,\theta^n,\psi^n)$ exists only on a time interval $[0,t_n]$, so that we have to guarantee that $t_n$ is bounded from below by some $t_0>0$ independent of $n$.

First we assume that 
\begin{equation}
\frac{m}{2}\leq v^n(y,t)\leq 2M,\hspace{10mm}y\in\Omega,t\in[0,t_n].\label{mvM4}
\end{equation}
This is certainly true on a possibly smaller time interval.

Second, the simplest case of Sobolev imbeddings applied to $\theta^n$ implies
\begin{equation}
\max_{y\in\Omega}|\theta^n(t,y)|\leq ||\theta^n(t)||_{L^2(\Omega)} + ||\theta_y^n(t,y)||_{L^2(\Omega)}.
\end{equation} 

Our growth conditions hypotheses on $p$ and $e$ then imply
\begin{equation}
\begin{cases}
0\leq p(v^n,\theta^n)\leq C(1+||\theta^n||_{L^2(\Omega)} + ||\theta_y^n||_{L^2(\Omega)}),\\
C_\vartheta(\theta^n)\geq C^{-1}, \hspace{10mm} \kappa(\theta^n)\geq C^{-1},\\
|p_\theta(v^n,\theta^n)|\leq C,\\
|p_v(v^n,\theta^n)|\leq C(1+||\theta^n||_{L^2(\Omega)} + ||\theta_y^n||_{L^2(\Omega)}),\\
|\kappa(\theta^n)| + |\kappa_\theta(\theta^n)| \leq C(1+||\theta^n||_{L^2(\Omega)}^{\tilde{q}} + ||\theta_y^n||_{L^2(\Omega)}^{\tilde{q}}).
\end{cases}\label{growth}
\end{equation}
Here, and in what follows, $C$ denotes a positive constant independent of $n$.

Finally, from (\ref{Lnv})we have that
\begin{align}
&m-t^{1/2}\left(\int_0^t||u_{yy}^n(s)||_{L^2(\Omega)}^2ds\right)^{1/2}\nonumber\\
&\hspace{20mm}\leq v^n(t,y)\leq M+t^{1/2}\left(\int_0^t||u_{yy}^n(s)||_{L^2(\Omega)}^2ds\right)^{1/2},\label{cotvn}
\end{align}
for a.e. $y\in\Omega$ and $t\in [0,t_n]$. Also,
\begin{equation}
||v_y^n(t)||_{L^2(\Omega)}\leq C+C\left(\int_0^t||u_{yy}^n(s)||_{L^2(\Omega)}^2ds\right)^{1/2},\label{cotvyn}
\end{equation}
for $t\in [0,\min \{t_n, 1 \}]$.

With these observations at hand, we can prove that as long as (\ref{mvM4}) holds (which is certainly true at $t=0$) we have the following inequality
\begin{equation}
\frac{d}{dt}\eta_n(t) \leq C_1(1+\eta_n(t)^{ q_1 }),\label{deseta}
\end{equation}
for a certain $q_1 >0$ where,
\begin{equation}
\eta_n(t)=||(u^n,\mathbf{w}^n,\mathbf{h}^n,\theta^n,\psi^n)(t)||_{H^1(\Omega)}^2 + \int_0^t ||(u_{yy}^n,\mathbf{w}_{yy}^n,\mathbf{h}_{yy}^n,\theta_{yy}^n)(s)||_{L^2(\Omega)}^2 ds.
\end{equation}
Since $\eta_n(0)$ is bounded by a constant $C_2>0$, then from (\ref{deseta}) we conclude that
\[
\eta_n(t) \leq \phi(t)
\]
for all $n=1,2,...$ and all $0\leq t < t^*$, where $\phi$ is the solution of the ODE
\begin{align*}
&\frac{d}{dt}X(t) = C_1(1+X(t)^{q_1})\\
&X(0)=C_2,
\end{align*}
and $t^*$ is the maximal time of existence of the solution to this ODE relative to the initial condition $C_2$.

After this, by (\ref{cotvn}) we can choose $0<t_0<t^*$ small enough so that (\ref{mvM4}) holds for all $t\in [0,t_0]$ and all $n$.

Now, \eqref{deseta} follows by multiplying (\ref{Lnu}) by $u^n-u_{yy}^n$, (\ref{LnSch}) by $\overline{\psi_t^n}$, (\ref{Lnw}) and (\ref{Lnh}) by $w^n-w_{yy}^n$ and $\mathbf{h}^n-\mathbf{h}_{yy}^n$ respectively and (\ref{LnE}) by $\theta^n-\theta_{yy}^n$, adding the resulting equations, integrating and using estimates \eqref{mvM4}-\eqref{cotvyn} combined with some tedious, but standard manipulation (which involves integration by parts and several applications of Young's inequality with $\varepsilon$). We omit the details.

In order to conclude the proof of the Lemma, a standard compactness argument implies that the sequence of approximate solutions converges to a solution of system \eqref{L1v}-\eqref{L1Sch}, defined on the time interval $[0,T]$, with $T=t_0$, and with all the desired properties.
\end{proof}

\subsection{A priori estimates}

The next step is to deduce some a priori estimates independent of time on the solutions of system \eqref{L1v}-\eqref{L1Sch} that allow us to extend the local solutions to global ones. The a priori estimates stated below are based on the analogues contained in \cite{CW} and in \cite{W}, on the study of the planar MHD equations. As the proofs are very similar we only give an outline of them and indicate the modifications that have to be made in order to include the coupling terms.

Let $(v,u,\mathbf{w},\mathbf{h},\theta,\psi)$ be a solution of \eqref{L1v}-\eqref{L1Sch}, \eqref{conds}. Let us assume that the solution is defined on a time interval $[0,T]$ where $T>0$ is fixed and that that $v(y,t),\theta(y,t)>0$ for all $(y,t)\in\Omega\times [0,T]$.

We begin with some energy estimates, followed by the uniform bounds from above and from below on the specific volume. In all of the subsequent calculations $C$ will denote a generic positive constant that may depend on $T$ and on the initial data. 

\begin{lemma}\label{apriori1}
\begin{equation}
\int_\Omega v(t,y)dy =\int_\Omega v_0(y)dy,\label{mass}
\end{equation}
\begin{align}
&\int_\Omega \Big(e+\theta^{1+r}+\tfrac{1}{2}(u^2+|\mathbf{w}|^2+\beta v|\mathbf{h}|^2) \nonumber\\
&\hspace{35mm}+  \alpha g(v)h(|\psi|^2) + \tfrac{1}{2}|\psi_y|^2 + \tfrac{1}{2}|\psi|^4 \Big)dy\leq C,\label{energy}
\end{align}
\begin{equation}
|\psi(t,y)|\leq C.\label{psiinfty}
\end{equation}
\end{lemma}

Estimate \eqref{mass} follows directly from equation \eqref{L1v} and the no-slip boundary condition $u|_{\partial \Omega}=0$ from \eqref{conds}, while \eqref{energy} follows from the energy equation \eqref{L1E}, our hypotheses on the initial data \eqref{cotas0}, \eqref{regL0} and the growth conditions \eqref{Q} on the internal energy.

Finally, \eqref{psiinfty} is a consequence of \eqref{energy} and the Sobolev embedding $H^1(\Omega)\hookrightarrow C(\Omega)$.

\begin{lemma}\label{apriori2}
\begin{equation}
C^{-1}\leq v(y,t)\leq C,\label{cVC}
\end{equation}
\begin{equation}
\int_\Omega(\theta-1-\log \theta)dy+\int_0^t\int_\Omega \frac{\kappa \theta_y^2}{v\theta^2}dy\hspace{0.5mm}ds\leq C,\label{kappatety}
\end{equation}
\begin{equation}
\int_0^t\int_\Omega \left( \varepsilon u_y^2+\mu|\mathbf{w}_y|^2+\nu|\mathbf{h}_y|^2+\theta_y^2\right)dy\hspace{0.5mm}ds\leq C.\label{uywyhy}
\end{equation}
\end{lemma}

Estimate \eqref{cVC} is based on an entropy estimate combined with an explicit form for the specific volume that can be deduced from the momentum equation \eqref{L1u}. More specifically, in view of \eqref{L1v}, equation \eqref{L1u} can be rewritten as
\begin{equation}
(\varepsilon \log v)_{yt}=u_t+\left(p+ \frac{\beta}{2}|\mathbf{h}|^2-\alpha g'(v)h(|\mathbf{w}|^2) \right)_y.\label{logv}
\end{equation}

Upon integration and with some manipulation we can get an explicit formula for $v$. After this, in connection with \eqref{entropy}, we consider the entropy $s=s(v,\theta)$ given by
\begin{equation}
s(v,\theta)=\int_1^\theta \frac{C_\vartheta(z)}{z}dz-\int_v^1 p_\theta(z)dz.\label{entropyv}
\end{equation}

Then, $s(v,\theta)$ satisfies the following equation
\begin{equation*}
s_t-\left( \frac{\kappa(\theta) \theta_y}{v\theta} \right)_y=\frac{\kappa(\theta)\theta_y^2}{v\theta^2}+\frac{\varepsilon u_y^2}{v\theta}+\frac{\mu |\mathbf{w}_y|^2}{v\theta}+\frac{\nu |\mathbf{h}_y|^2}{v\theta},\label{L1s}
\end{equation*}
and integrating this equation we obtain estimate \eqref{kappatety}, which provides a way of estimating the thermal part of the pressure in the process of obtaining \eqref{cVC}. In particular, using \eqref{K} we have
\begin{align}
\int_0^tM_\theta(s)ds&\leq \int_0^t\left( \int_\Omega\theta dy + \int_\Omega |\theta_y|dy \right)ds\nonumber\\
                     &\leq C+ \int_0^t\left( \int_\Omega \frac{\theta_y^2}{v}dy+\int_\Omega v\hspace{0.5mm}dy \right)ds\nonumber\\
                     &\leq C+C\int_0^t\int_\Omega \frac{\kappa \theta_y^2}{v\theta^2}dy\hspace{0.5mm}ds\nonumber\\
                     &\leq C,\label{Mtetint} 
\end{align}
where $M_\theta(t)=\max_{y\in\Omega}\theta(t,y)$. Note that here we use our hypothesis \eqref{K} on the heat conductivity $\kappa$.

Finally, having \eqref{cVC}, estimate \eqref{uywyhy} results by integrating equation \eqref{L1Q}.

At this point, the estimates from Lemma \ref{apriori1} on the wave function $\psi$, as well as our hypotheses \eqref{gh} on the coupling functions, suffice in order to deal with the coupling term in equation \eqref{logv} without any major difficulties.

We continue with an estimate on the $L^2$ norm of the derivatives of $v,u,\mathbf{w}$ and $\mathbf{h}$.

\begin{lemma}\label{apriori3}
\begin{align}
&\int_\Omega (v_y^2+u_y^2+|\mathbf{w}_y|^2+|\mathbf{h}_y|^2)dy\nonumber\\
&\hspace{14mm}+\int_0^t\int_\Omega(u_{yy}^2+|\mathbf{w}_{yy}|^2+|\mathbf{h}_{yy}|^2+u_y^4+|\mathbf{w}_y|^4+|\mathbf{h}_y|^4)dy\hspace{0.5mm}ds\leq C.
\end{align}
\begin{equation}
\int_\Omega (|\psi_{t}|^2 + |\psi_{yy}|^2)dy\leq C,\label{psiyy}
\end{equation}
\end{lemma}

\begin{proof}
First we deal with the $L^2$ norm of $v_y$. Define $V(y,t):=\varepsilon\log v$. Then, from \eqref{L1v} we see that $V$ satisfies $V_t=\varepsilon \frac{u_y}{v}$, so we can rewrite equation \eqref{L1u} as 
\begin{equation}
(V_y-u)_t = \left( p+\frac{\beta}{2}|\mathbf{h}|^2 - \alpha g'(v)h(|\psi|^2) \right)_y.\label{L1V}
\end{equation}

Multiply the above equation by $V_y-u$ and integrate. After some manipulation and using Lemma \ref{apriori1} we obtain
\begin{flalign*}
&\frac{1}{2}\int_\Omega |V_y -u|^2dy&\\
&\hspace{10mm}\leq C+C\int_0^t(1+M_\theta(s))\int_\Omega |V_y-u|^2dy\hspace{0.5mm}ds+ C\int_0^t\int_\Omega |\mathbf{h}\cdot\mathbf{h}_y|^2\hspace{0.5mm}ds. &
\end{flalign*}

And since we have \eqref{Mtetint}, Gronwalls's inequality then implies
\begin{equation*}
\int_\Omega |V_y -u|^2dy\leq C+C\int_0^t\int_\Omega |\mathbf{h}\cdot\mathbf{h}_y|^2\hspace{0.5mm}ds.
\end{equation*}

In particular,
\begin{equation}
\int_\Omega v_y^2dy\leq C+C\int_0^t\int_\Omega |\mathbf{h}\cdot\mathbf{h}_y|^2\hspace{0.5mm}ds.\label{vy-}
\end{equation}

Following \cite{CW,W}, in order to bound the right hand side of this inequality, we multiply equation \eqref{L1h} by $|\mathbf{h}|^2\mathbf{h}$, integrate over $\Omega\times [0,t]$ and use \eqref{L1v} to obtain
\begin{align*}
&\frac{1}{4}\int_\Omega v|\mathbf{h}|^4dy +\int_0^t \int_\Omega \nu v^{-1}(2|\mathbf{h}\cdot\mathbf{h}_y|^2+|\mathbf{h}|^2|\mathbf{h}_y|^2) dy \hspace{.5mm}ds\\
&\hspace{5mm}=-\frac{3}{4}\int_0^t\int_\Omega u_y|\mathbf{h}|^4dy\hspace{.5mm}ds + \int_0^t \int_\Omega \mathbf{w}_y\cdot|\mathbf{h}|^2\mathbf{h} dy \hspace{.5mm}ds + \frac{1}{4}\int_\Omega v_0|\mathbf{h}_0|^2dy,
\end{align*}
wherein, after some manipulation we get
\begin{equation}
\int_\Omega |\mathbf{h}|^4dy +\int_0^t \int_\Omega |\mathbf{h}\cdot\mathbf{h}_y|^2 dy \hspace{.5mm}ds \leq C.
\end{equation}

Thus, concluding by \eqref{vy-} that
\begin{equation}
\int_\Omega v_y^2dy\leq C.
\end{equation}

After this, the stated bounds on $u$, $\mathbf{w}$ and $\mathbf{h}$ are standard estimates on the parabolic equations and follow by multiplying \eqref{L1u}, \eqref{L1w} and \eqref{L1h} by $u_{yy}$, $\mathbf{w}_{yy}$ and $\mathbf{h}_{yy}$, respectively, integrating by parts and using Young's inequality as usual.

Finally, we are left with \eqref{psiyy}. For this, we differentiate equation \eqref{L1Sch} with respect to $t$, multiply it by $\overline{\psi}_t$ (the complex conjugate of $\psi_t$), take imaginary part and integrate to obtain
\begin{align*}
\int_\Omega |\psi_t|^2 dy &\leq C+C\int_0^t\int_\Omega v_t^2 dy\hspace{.5mm}ds + C\int_0^t\int_\Omega |\psi_t|^2 dy\hspace{.5mm}ds\\
  &= C+C\int_0^t\int_\Omega u_y^2 dy\hspace{.5mm}ds + C\int_0^t\int_\Omega |\psi_t|^2 dy\hspace{.5mm}ds\\
  &\leq C+C\int_0^t\int_\Omega |\psi_t|^2 dy\hspace{.5mm}ds,
\end{align*}
and from Gronwall's inequality we get
\begin{equation}
\int_\Omega |\psi_t|^2 dy\leq C.
\end{equation}

Note that in light of all the estimates we have deduced so far, and in view of equation \eqref{L1Sch}, the $L^2(\Omega)$-norm of $\psi_t$ is equivalent to the $L^2(\Omega)$-norm of $\psi_{yy}$ (it is at this point that we use our assumption that $\psi_0\in H^2(\Omega)$). Thus we conclude that
\begin{equation}
\int_\Omega |\psi_{yy}|^2 dy\leq C.
\end{equation}
\end{proof}

We now turn our attention to the a priori estimates on the derivatives of the temperature.

\begin{lemma}\label{apriori4}
\begin{equation}
C^{-1}\leq\theta\leq C,\label{tetC}
\end{equation}
\begin{equation}
|v_y|\leq C,
\end{equation}
\begin{equation}
\int_\Omega \theta_y^2dy + \int_0^t\int_\Omega(\theta_t^2+\theta_{yy}^2)dy\hspace{.5mm}ds\leq C.\label{tetytetyy}
\end{equation}
\end{lemma}

\begin{proof}
Following \cite{CW,W}, we set
\begin{align*}
&\Theta:=\max_{(y,t)\in\Omega\times [0,t]}\theta(y,t), & X:=\int_0^T\int_\Omega(1+\theta^{q+r})\theta_t^2dy\hspace{.5mm}ds\\
& Y:=\max_{t\in[0,T]}\int_\Omega(1+\theta^{2q})\theta_y^2 dy. & 
\end{align*}

Using \eqref{K} we have
\begin{align*}
\max_{y\in\Omega}\theta^{(2q+3+r)/2}&\leq \left(\int_\Omega \theta dy\right)^{(2q+3+r)/2}+\frac{2q+3+r}{2}\int_\Omega\theta^{(2q+1+r)/2}|\theta_y|dy\\
             &\leq C+C\left(\int_\Omega \theta^{1+r} dy\right)^{1/2}\left(\int_\Omega \theta^{2q}\theta_y^2 dy\right)^{1//2}\\
             &\leq C+CY^{1/2}.
\end{align*}
Thus,
\begin{equation}
\Theta\leq C+CY^{\delta_1}\label{boundtetY}
\end{equation}
where $\delta_1=(2q+3+r)^{-1}$.

The next step is to show that $X+Y\leq C$. For this, we define $H(v,\theta):=v^{-1}\int_0^\theta \kappa(\xi)d\xi$, so that
\begin{align*}
&H_t=H_v u_y + \frac{\kappa}{v}\theta_t,\\
&H_{ty}=H_v u_{yy} + H_{vv}u_y v_y+ \left(\frac{1}{v}\right)_v\kappa\theta_t v_y + \left(\frac{\kappa}{v}\theta_y\right)_t.
\end{align*}

Rewriting equation \eqref{L1Q} as
\begin{equation}
C_\vartheta(\theta) \theta_t + \theta p_\theta u_y =  \left(\frac{\kappa\theta_y}{v} \right)_y +\frac{\varepsilon u_y^2}{v} + \frac{\mu |\mathbf{w}_y|^2}{v} + \frac{\nu |\mathbf{h}_y|^2}{v}, \label{L1tet}
\end{equation}
we multiply this equation by $H_t$ and integrating by parts we get
\begin{flalign}
&\int_0^t \int_\Omega \Big(C_\vartheta \theta_t + \theta p_\theta u_y -\frac{\varepsilon u_y^2}{v} + \frac{\mu |\mathbf{w}_y|^2}{v} + \frac{\nu |\mathbf{h}_y|^2}{v}\Big)H_t dy\hspace{.5mm}ds &\nonumber\\
&\hspace{60mm}+\int_0^t \int_\Omega \frac{\kappa\theta_y}{v}H_{ty}dy\hspace{.5mm}ds=0.&\label{TetXY}
\end{flalign}

At this point, we observe that from \eqref{Q} we have
\[
\int_0^t\int_\Omega  C_\vartheta \theta_t \frac{\kappa}{v}\theta_t dy\hspace{.5mm}ds \geq M_1 X,
\]
and also
\[
\int_0^T\int_\Omega \frac{\kappa\theta_y}{v}\left(\frac{\kappa}{v}\theta_y\right)_t dy\hspace{.5mm}ds=\int_0^T\frac{d}{dt}\int_\Omega \frac{\kappa^2\theta_y^2}{v^2} dy\hspace{.5mm}ds\geq M_2 Y - C,
\]
for positive constants $M_1$ and $M_2$. As a consequence, in order to show that $X+Y\leq C$ all that is left to do is bound appropriately the rest of the terms in \eqref{TetXY}. Although this is not simple, it follows the same lines as in \cite{CW,W}, and we omit the details.

Now, this together with \eqref{boundtetY} leads to the estimates
\begin{equation}
\theta\leq C,
\end{equation}
and
\begin{equation}
\int_\Omega \theta_y^2dy+\int_0^t\int_\Omega \theta_t^2dy\hspace{.5mm}ds\leq C.
\end{equation}

We finally move on to the last estimate, consisting of a lower bound for the temperature. We have to prove that 
\begin{equation}
C^{-1}\leq \theta(y,t),\label{ctet}
\end{equation}
for a big enough constant $C>0$. In order to show this it suffices to apply the maximum principle (see \cite{PWe}) to equation \eqref{L1tet}. More specifically, note that $\theta$ satisfies the following inequality
\begin{equation}
C_\vartheta \theta_t + \frac{v}{2\varepsilon}\theta^2 p_\theta^2 -\left(\frac{\kappa}{v}\right)_v \theta_y v_y\geq \frac{\kappa\theta_{yy}}{v}. \label{ineqtet}
\end{equation}
This follows from equation \eqref{L1tet} and Young's inequality. In order to apply the maximum principle we have to show that the coefficients of this parabolic inequality are bounded. With the estimates already obtained, we only have to show that $v_y$ is uniformly bounded. Let $V$ be as in \eqref{L1V}. Then,
\begin{align*}
&V_y(y,t)=V_{y}(y,0)+u(y,t)-u_0(y)+\int_0^t(p_v v_y + p_\theta \theta_y +\beta\mathbf{h}\cdot\mathbf{h}_y\\
&\hspace{40mm}-\alpha g''(v)h(|\psi|^2)v_y-2\alpha g'(v)h'(|\psi|^2)\text{Re}(\psi\overline{\psi}_y))ds.
\end{align*}

Squaring this identity and using interpolation inequalities on $\theta_y$ $\mathbf{h}_y$ and $\psi_y$ we get
\begin{align*}
v_y^2&\leq C+ C\int_0^t\int_\Omega(|\theta_{yy}^2+|\mathbf{h}_{yy}|^2+|\psi_{yy}|^2)dy\hspace{.5mm}ds + \int_0^t v_y^2ds\\
   &\leq C+C\int_0^t v_y^2ds,
\end{align*}
which yields, using Gronwall's inequality that
\[
|v_y|\leq C
\].

Consequently, taking into consideration the boundary conditions on $\theta$,  the maximum principle  implies that $\theta$ cannot be zero in finite time, which concludes the proof.
\end{proof}

The estimates from Lemmas \ref{apriori1} through \ref{apriori4} provide the necessary a priori estimates which, in light of the local result from Lemma \ref{local}, guarantee the gobal existence of solutions. The uniqueness of solutions can be carried out in a straightforward way by using some energy estimates similar to those above in combination with Gronwall's inequality applied to the subtraction of the equations satisfied by two possible solutions with the same initial data.

To conclude the proof of Theorem \ref{thm2}, we note that by \eqref{Lag3} and \eqref{cVC}, we can turn back to the original Eulerean coordinates.

\section{Limit equations}\label{limit}

We now move on to the vanishing viscosity problem. Our objective here is to study the limit of solutions of system \eqref{E1rho}-\eqref{E1Sch} as $(\varepsilon,\alpha,\beta,\delta)\to 0$. In order to deal with the convergence issues in the continuity and momentum equations we adapt the framework by Chen and Perepelitsa in \cite{CP}. Regarding the convergence of solutions in the nonlinear Schr\"{o}dinger equation a simple application of Aubin-Lions lemma will suffice. The magnetic description poses no problems as we can deduce good uniform estimates on the magnetic field. Lastly, for the thermal description we adapt some ideas from the study by Feireisl in \cite{Fe} on the full multidimensional compressible Navier-Stokes equations.

In the interest of analysing the limit as $(\varepsilon,\alpha,\beta,\delta)\to 0$, let us make some considerations on the limit equations. In order to fix notation we denote by $\Omega$ the spatial domain where the Eulearean coordinates take values and by $\Omega_y$ the corresponding domain of the Lagrangian coordinate. We begin by the compressible Euler equations.

\subsection{Isentropic Euler equations}

Let us consider the isentropic Euler equations 
\begin{flalign}
 &\rho_t+(\rho u)_x=0,&\label{Eulerrho},\\
 &(\rho u)_t + (\rho u^2 + p_e(\rho) )_x=0 ,&\label{Euleru}
\end{flalign}
where the pressure $p_e(\rho)$ is given by
\begin{equation}
p_e(\rho)=a\rho^\gamma\label{press}
\end{equation}
for some $\gamma>1$, with initial data
\begin{equation}
(\rho(x,0),u(x,0))=(\rho_0(x),u_0(x)).\label{Euler0}
\end{equation}
As it is not possible to avoid the occurrence of vacuum  in this setting, it is convenient to consider the momentum $m=\rho u$ as state variable in place of the velocity. Accordingly, system \eqref{Eulerrho}, \eqref{Euleru} may be written in the form of a nonlinear hyperbolic system of conservation laws
\begin{equation}
U_t+F(U)_x=0\label{conslaw}
\end{equation}
where, $U=(\rho,m)^\top$ and $F(U)=(m,\frac{m^2}{\rho}+p_e)$.

Let us recall that a pair of functions $(\eta,q):\mathbb{R}_+^2\to \mathbb{R}^2$ is called an {\sl entropy-entropy flux pair} (or simply entropy pair) of system \eqref{conslaw} provided that they satisfy
\begin{equation}
\nabla q(\mathbf{U})=\nabla\eta(\mathbf{U})\nabla F(\mathbf{U}).\label{entropy'}
\end{equation}
An entropy pair for \eqref{Eulerrho}, \eqref{Euleru} is said to be convex if the Hessian $\nabla^2\eta(\rho,m)\geq 0$, and $\eta$ is called a {\sl weak entropy} if
\[
\lim_{\substack{\rho\to 0,\\\frac{m}{\rho}}=\text{const.}}\eta(\rho,m)=0.
\]
A very important example of weak entropy pair for \eqref{Eulerrho}, \eqref{Euleru} is given by the {\sl mechanical energy} $\eta^*$ and the {\sl mechanical energy flux} $q^*$:
\begin{equation}
\eta^*(\rho,m)=\frac{1}{2}\frac{m^2}{\rho}+\rho P_e(\rho),\hspace{10mm}q^*(\rho,m)=\frac{1}{2}\frac{m^3}{\rho^2}+mP_e(\rho)+\rho m P_e'(\rho),
\end{equation}
where, as before, $P_e(\rho)$ is the elastic potential
\[
P_e(\rho)=\frac{a}{\gamma-1}\rho^{\gamma-1}.
\]
The {\sl total mechanical energy} for \eqref{Eulerrho}, \eqref{Euleru} is
\begin{equation}
E[\rho,u](t):=\int_\Omega \eta^*(\rho,m)dx=\int_\Omega \left( \frac{1}{2}\rho u^2 + \frac{a}{\gamma-1}\rho^\gamma \right) dx.\label{totmechenergy}
\end{equation}

Relation \eqref{entropy'} may be written in the variables $(\rho,u)$ as the wave equation
\begin{equation}
\begin{cases}
\eta_{\rho\rho}-\frac{p_e'(\rho)}{\rho^2}\eta_{uu}=0, &\rho>0\\
\eta|_{\rho=0}=0. & 
\end{cases}
\end{equation}
and consequently, any weak entropy pair $(\eta,q)$ can be represented by
\[
\begin{cases}
\eta^\zeta(\rho,\rho u)=\int_\mathbb{R}\chi(\rho;s-u)\zeta(s)ds, & \\
q^\zeta(\rho,\rho u)=\int_\mathbb{R} (\vartheta s +(1+\vartheta)u)\chi(\rho;s-u)\zeta(s)ds, &\vartheta=\frac{\gamma-1}{2},
\end{cases}
\]
for any continuous function $\zeta(s)$, where $\chi(\rho,u;s)=\chi(\rho;s-u)$ is determined by
\begin{equation}
\begin{cases}
\chi_{\rho\rho}-\frac{p_e'(\rho)}{\rho^2}\chi_{uu}=0, &\\
\chi(0,u;s)=0, &\chi_\rho(0,u;s)=\delta_{u=s}. 
\end{cases}
\end{equation}

For the $\gamma$-law case, where the pressure is given by \eqref{press}, the weak entropy kernel is given by
\begin{equation}
\chi(\rho,u;s)=[\rho^{2\vartheta}-(s-u)^2]_+^\Lambda,\hspace{10mm}\Lambda=\frac{3-\gamma}{2(\gamma-1)},\label{entrkernel}
\end{equation}
and the corresponding weak entropy pairs are given by
\begin{equation}
\begin{cases}
\eta^\zeta(\rho,\rho u)=\rho\int_{-1}^1 \zeta(u+\rho^{\vartheta}s)[1-s^2]_+^\Lambda ds, & \\
q^\zeta(\rho,\rho u)=\rho\int_{-1}^1 (u+\vartheta\rho^\vartheta s)\zeta(u+\rho^{\vartheta}s)[1-s^2]_+^\Lambda ds. &
\end{cases}\label{entropexpl}
\end{equation}

A direct consequence of this representation is the following (\cite[Lemma 2.1]{CP}), which we state for later reference.
\begin{lemma}\label{boundsetaq}
For a $C^2$ function $\zeta:\mathbb{R}\to\mathbb{R}$ compactly supported in $[a,b]$, we have
\[
\text{supp} \eta^\zeta,\text{ supp}q^\zeta\subseteq \{ (\rho,\rho u):\rho^\vartheta+u\geq a, u-\rho^\vartheta\leq b \}.
\]
Furthermore, there exists a constant $C_\zeta>0$ such that, for any $\rho\geq 0$ and $u\in\mathbb{R}$, we have
\begin{itemize}
\item{(i)} For $\gamma\in (1,3]$,
\[
|\eta^\zeta(\rho,m)|+|q^\zeta(\rho,m)|\leq C_\zeta\rho.
\]
\item{(ii)} For $\gamma>3$,
\[
|\eta^\zeta(\rho,m)|\leq C_\zeta \rho,\hspace{10mm}|q^\zeta(\rho,m)|\leq C_\zeta\rho\max\{1,\rho^\vartheta \}.
\]
\item[(iii)] If $\eta_n^\zeta$ is considered as a function of $(\rho,m)$, $m=\rho u$, then
\[
|\eta_{m}^\zeta(\rho,m)|+|\rho\eta_{mm}^\zeta(\rho,m)|\leq C_\zeta,
\]
and if $\eta_n^\zeta$ is considered as a function of $(\rho,u)$, then
\[
|\eta_{mu}^\zeta(\rho,m)|+|\rho^{1-\vartheta}\eta_{m\rho}^\zeta(\rho,m)|\leq C_\zeta.
\]
\end{itemize} 
\end{lemma}

In light of these considerations we state the following concept, taken from \cite{CP}.

\begin{definition}
Let $(\rho_0,u_0)$ be given initial data such that $E[\rho_0,u_0]\leq E_0<\infty$. A pair $(\rho,u):\Omega\times[0,T)\to[0,\infty)\times\mathbb{R}$ is called a  finite-energy entropy solution of \eqref{Eulerrho},\eqref{Euleru},\eqref{Euler0} if the following hold:
\begin{itemize}
\item There is a locally bounded function $C(E,t)\geq 0$ such that
\[
E[\rho,u](t)\leq C(E_0,t).
\]
\item $(\rho,u)$ satisfies \eqref{Eulerrho} and \eqref{Euleru} in the sense of distributions and, more generally, 
\[
\eta^\zeta(\rho,u)_t+q^\zeta(\rho,u)_x\leq 0,
\]
in the sense of distributions, for the test functions $\zeta(s)\in\{ \pm 1, \pm s, s^2 \}$.
\item The initial data are attained in the sense of distributions.
\end{itemize}
\end{definition}

The first step in our analysis is to show strong convergence of a subsequence of $(\rho^\varepsilon,\rho^\varepsilon u^\varepsilon)$ to a finite-energy entropy solution to \eqref{Eulerrho},\eqref{Euleru}. For this we adapt the compactness scheme in \cite{CP}, which is based on the compensated compactness method (cf. \cite{T,T',Mu,DP1,DP,C,C',DC,LPS,LPT,LW}).

\subsection{Transverse velocity field and magnetic field}

We move on to the limit equations \eqref{E1winfty} and \eqref{E1hinfty} for the transverse velocity field and the magnetic field. As $\mu$ and $\nu$ are left fixed independently of $\varepsilon$ we can deduce some satisfactory uniform estimates on $\mathbf{w}_x$ and on $\mathbf{h}_x$ that permit the passage to the limit for the sequence $(\mathbf{w}^\varepsilon,\mathbf{h}^\varepsilon)$ without any major complications, once we have shown that $\rho^\varepsilon$ and $\rho^\varepsilon u^\varepsilon$ converge strongly.

Regarding equation \eqref{E1hinfty}, we see that we are left with a stationary equation and therefore the initial condition loses its meaning. However, note that from equation \eqref{E1h} we have that
\[
\int_\Omega \beta\mathbf{h}^\varepsilon\varphi dx-\int_\Omega \beta \mathbf{h}_0^\varepsilon \varphi dx - \int_0^t\int_\Omega(\beta u^\varepsilon\mathbf{h}^\varepsilon - \beta\mathbf{w}^\varepsilon)\varphi_x dx\hspace{.5mm}ds =-\int_0^t\int_\Omega \nu\mathbf{h}_x^\varepsilon \varphi_x dx\hspace{.5mm}ds,
\] 
for any smooth test function $\varphi=\varphi(x)$ with compact support in $\Omega$.

A couple of energy estimates based on the energy identity \eqref{difE1} and on equation \eqref{E1Q} (which will be deduced later) as well as an interpolation inequality for $u^\varepsilon$ and for $\mathbf{w}^\varepsilon$ show that $\beta \mathbf{h}^\varepsilon\to 0$ in $L^\infty(0,T;L^2(\Omega))$ and $\beta(u^\varepsilon\mathbf{h}^\varepsilon,\mathbf{w}^\varepsilon)\to 0$ in $L^1(\Omega\times (0,T))$ as $\varepsilon\to 0$ with $\beta=o(\varepsilon)$. By the same token we can a assume that $\mathbf{h}_x^\varepsilon$ converges weakly to $\mathbf{h}_x$ for some $\mathbf{h}\in L^2(0,T;H_0^1(\Omega))$. As a result, in the limit we have
\[
\lim_{\varepsilon\to 0}\int_\Omega \beta\mathbf{h}_0^\varepsilon \varphi ds  =\int_0^t\int_\Omega \nu\mathbf{h}_x \varphi_x dx\hspace{.5mm}ds.
\] 

For this reason, we are compelled to impose that $\beta \mathbf{h}_0^\varepsilon \to 0$ in the sense of distributions, in which case we would have that $\mathbf{h}_x=0$, thus forcing $\mathbf{h}$ to be identically equal to zero.

As for the limit equation \eqref{E1winfty}, the same energy estimates allow us to assume that $\mathbf{w}^\varepsilon\to \mathbf{w}$ weakly in $L^2(0,T;H_0^1(\Omega))$ and provided that $\rho^\varepsilon$ and $\rho^\varepsilon u^\varepsilon$ converge strongly, we can conclude that the limit equation \eqref{E1winfty} is satisfied in the sense of distributions.  

\subsection{Thermal energy}

The uniform estimates that we obtain further ahead, guarantee that $\varepsilon |u^\varepsilon|^2$, $\mu|\mathbf{w}^\varepsilon|^2$ and $\nu|\mathbf{h}^\varepsilon|^2$ are bounded in $L^1(\Omega\times(0,T))$. Nonetheless, this is the best uniform estimate that we can hope to obtain on the derivatives of $u$, $\mathbf{w}$ and $\mathbf{h}$. This means that, consistency becomes an issue in the thermal energy limit equation \eqref{E1Qinfty} as we cannot guarantee that the sequence (or any subsequence of) $\varepsilon |u^\varepsilon|^2+\mu|\mathbf{w}^\varepsilon|^2+\nu|\mathbf{h}^\varepsilon|^2$ converges to anything other than possibly a positive Radon measure. For this reason we do not expect equation \eqref{E1Qinfty} to be satisfied and the best we can aim to obtain when taking the limit as $\varepsilon\to 0$ in equation \eqref{E1Q} is an inequality. 

On the bright side we note that given a nonnegative smooth test function $\varphi$, the function $f\to \int_0^t\int_\Omega |f|^2\varphi dx\hspace{.5mm}ds$ defined for $f\in L^2(\Omega\times (0,T))$ and taking values in $[0,\infty)$ may be regarded as the squared norm in the weighted $L_\varphi^2$ space. As we have that the sequence $(\mathbf{w}_x^\varepsilon,\mathbf{h}_x^\varepsilon)$ is weakly convergent in $L^2(\Omega\times(0,T))$ (and therefore, also in $ L_\varphi^2(\Omega\times(0,T))$) we see that 
\[
\liminf_{\varepsilon\to 0}\int_0^t\int_\Omega (\mu|\mathbf{w}_x^\varepsilon|^2 + \nu|\mathbf{h}_x^\varepsilon|^2)\varphi dx \hspace{.5mm} ds \geq \int_0^t\int_\Omega (\mu|\mathbf{w}_x|^2 + \nu|\mathbf{h}_x|^2)\varphi dx ds.
\]

Recall that, in fact, the limit magnetic field $\mathbf{h}$ has to be equal to zero. With this in mind, we intend to show that, in the limit, the following inequality
\begin{equation}
(\rho Q(\theta))_t + (\rho Q(\theta)\mathbf{u})_x \geq (\kappa \theta_x)_x + \mu |\mathbf{w}_x|^2,\label{E1Qineq}
\end{equation}
is satisfied in the sense of distributions by the limit functions. In the process we are going to show that the following inequality also holds
\begin{align}
&\int_\Omega \rho\left(P_e(\rho) + Q(\theta) + \frac{1}{2}|u|^2+\frac{1}{2}|\mathbf{w}|^2 \right)(t)dx\nonumber\\
&\hspace{40mm}\leq \int_\Omega \left(\rho_0P_e(\rho_0) + \rho_0Q(\theta_0) + \frac{1}{2}\frac{m_0^2}{\rho_0}+\rho_0\frac{1}{2}|\mathbf{w}|^2 \right)dx.\label{E1Einftyweak}
\end{align}

This is nothing other than to say, in the notation of \eqref{totmechenergy}, that
\[
E[\rho,u](t)+||(\rho Q(\theta),\rho|\mathbf{w}|^2)(t)||_{L^1(\Omega)}\leq E[\rho_0,u_0](0)+||(\rho_0 Q(\theta_0),\rho_0|\mathbf{w_0}|^2)||_{L^1(\Omega)},
\]
which compensates, in some way, the ``loss of information'' resulting from considering an inequality instead of an identity in the limit thermal energy equation.

With this in mind we state the following definition.
\begin{definition}
We say that $(\rho,u,\mathbf{w},\mathbf{h},\theta)$ constitute a varional solution of equation \eqref{E1Qinfty} with initial data $\theta|_{t=0}=\theta_0$ provided that it satisfies both 
\begin{flalign}
&\int_\Omega \left(\rho\left(e(\rho,\theta) +\tfrac{1}{2}|u|^2 + \tfrac{1}{2}|\mathbf{w}|^2  \right)+\tfrac{\beta}{2}|\mathbf{h}|^2\right)(t)  dx \nonumber\\
&\hspace{25mm}\leq \int_\Omega \rho_0\left(e(\rho_0,\theta_0) +\tfrac{1}{2}|u_0|^2 + \tfrac{1}{2}|\mathbf{w}_0|^2+\tfrac{\beta}{2}|\mathbf{h}_0|^2  \right)  dx.&&\label{defenergyfinal}
\end{flalign}
and
\begin{flalign}
&\int_0^T \int_\Omega \big(\rho Q(\theta)\varphi_t + \rho u Q(\theta)\varphi_x + \mathcal{K}(\theta)\varphi_{xx}\big) dx ds &\nonumber\\
&\hspace{15mm}\leq -\int_0^T\int_\Omega (\mu|\mathbf{w}_x|^2\varphi +\nu |\mathbf{h}_x|^2)dxds -\int_\Omega \rho_0 Q(\theta_0)\varphi|_{t=0}dx,&\label{defE1Qinftyineq}
\end{flalign}
for any test function $\varphi$ such that
\begin{equation}
\varphi\geq 0,\hspace{5mm} \varphi\in W^{2,\infty}(\Omega\times(0,T)),\hspace{5mm} \psi_x|_{\partial \Omega}=0,\hspace{5mm} \text{supp}\varphi\subseteq \overline{\Omega}\times[0,T).\label{deftest}
\end{equation}
\end{definition}

This is in accordance with the definition of variational solution of the thermal energy equation considered by Feireisl in \cite{Fe} (see \cite[Definition 4.5]{Fe}).

Let us point out, that even by considering the inequality \eqref{E1Qineq} in place of \eqref{E1Qinfty}, the task of showing consistency is not simple as $Q$ and $\kappa$ are nonlinear functions of $\theta$. This means that we have to show strong convergence of the sequence $\theta^\varepsilon$. 

For this we adapt an idea in \cite{Fe} which can be divided into two steps. First, using uniform estimates and some careful analysis we can show that $Q(\theta^\varepsilon)$ converges pointwise to some limit $\overline{Q}$, in the set where $\rho$ (the limit density) is positive. As $Q$ is a strictly increasing function, we can write $\overline{Q}$ as $\overline{Q}=Q(\hspace{.5mm}\overline{\theta}\hspace{.5mm})$, i.e., $\overline{\theta}=Q^{-1}(\overline{Q})$. Then, using \eqref{Q} we see that
\begin{align*}
0&=\lim_{\varepsilon\to 0}\int_0^T\int_\Omega (Q(\theta^\varepsilon)-Q(\overline{\theta}))(\theta^\varepsilon-\overline{\theta})\mathbbm{1}_{\{\rho>0\}}dx\hspace{.5mm}ds\\
&\geq \lim_{\varepsilon\to 0}C^{-1}\int_0^T\int_\Omega(\theta^\varepsilon-\overline{\theta})^2\mathbbm{1}_{\{\rho>0\}}dx\hspace{.5mm}ds,
\end{align*}
so that $\theta^\varepsilon$ also converges pointwise to $\overline{\theta}$ in the set $\{\rho>0\}$. After this, we adapt a clever argument from \cite{Fe} to show that the function $\mathcal{K}(\theta):=\int_0^\theta \kappa(z)dz$ converges weakly to some $\overline{\mathcal{K}}$. Accordingly, $\overline{\mathcal{K}}=\mathcal{K}(\overline{\theta})$ in the set where $\rho>0$. Thus, if we define $\theta:=\mathcal{K}^{-1}(\hspace{.5mm}\overline{\mathcal{K}}\hspace{.5mm})$ we then have that $\theta=\overline{\theta}$ in the set $\{\rho>0\}$ and we can pass to the limit in equation \eqref{E1Q} in order to conclude that $\theta$ satisfies inequality \eqref{E1Qineq}.

We will fill in the details of this procedure later.

\subsection{Nonlinear Schr\"{o}dinger equation}

Finally, we consider the limit equation \eqref{E1Schinfty}. Let us recall that $\psi:\Omega_y\times(0,T)\to\mathbb{C}$ is called a weak solution of \eqref{E1Schinfty} with initial data $\psi|_{t=0}=\psi_0$ if $\psi\in L^\infty(0,T;H_0^1(\Omega_y))\cap W^{1,\infty}(0,T;H^{-1}(\Omega_y))$, \eqref{E1Schinfty} is satisfied in $H^{-1}(\Omega_y)$ for each $t\in (0,T)$ and the initial data is attained in the sense of distributions. Existence and uniqueness of global weak solutions to \eqref{E1Schinfty} with initial data $\psi_0\in H^1(\Omega)$ is a well known result (see \cite{K},\cite{Ca}).

Assuming, as before, that $\alpha=o(\varepsilon^{1/2})$, the energy identity \eqref{difE1} yields uniform estimates on the $L^\infty(0,T;L^4(\Omega_y)\cap H_0^1(\Omega_y))$ norm of $\psi^\varepsilon$. In view of our hypotheses \eqref{gh} on the coupling, functions a direct application of Aubin-Lions lemma (see \cite{JLi,Au,Si}) allows us to pass to the limit in equation \eqref{L1Sch}.

As for the initial data, we only have to assume that $\psi_0^\varepsilon\to\psi_0$ in $H_0^1$ as $\varepsilon \to 0$ for the argument above to hold.

\section{Uniform estimates}\label{estimates}

Our goal now, is to deduce some uniform estimates that allow us to proceed as sketched above. They are divided into several lemmas. Lemmas \ref{uniform1} through \ref{uniform4} are inspired by their analogues contained in \cite{CP}, although with several improvements in order to include the thermal description, the magnetic field and the coupling terms. To avoid the overload of notation, in this Section we denote by $(\rho,u,\mathbf{w},\mathbf{h},\theta,\psi)$ a solution of \eqref{E1rho}-\eqref{E1Sch} with initial conditions $(\rho_0,u_0,\mathbf{w}_0,\mathbf{h}_0,\theta_0,\psi_0)$. We also assume, without loss of generality that $\Omega=(0,1)$. The estimates below are uniform in the sense that the bounding constants do not depend on $\varepsilon$ (and hence nor on $\beta$, $\alpha$ or $\delta$). To this end, in what follows $C$ will stand for a universal constant independent of $\varepsilon$. We also assume that $\alpha=o(\varepsilon^{1/2})$, $\beta=0(\varepsilon)$ and $\delta=o(\varepsilon)$ as $\varepsilon\to 0$, and that $\mu$ and $\nu$ are fixed positive constants independent of $\varepsilon$ and that $\kappa$ satisfies \eqref{K}, also independently of $\varepsilon$.

We begin with the following basic energy estimate.

\begin{lemma}\label{uniform1}
Let $\delta=o(\varepsilon)$. Assume that 
\[
C_0^{-1}\leq \int_0^1 \rho_0 dx\leq C_0,\hspace{10mm} -\int_\Omega \rho_0 s(\rho_0,\theta_0)dx\leq C_0
\]
where $s$ is the entropy given by \eqref{entropy}, and that
\begin{align*}
&\int_\Omega\Big(\rho_0\Big(e(\rho_0,\theta_0)+ \frac{1}{2}_0u^2+\frac{1}{2}|\mathbf{w}_0|^2 \Big) +\frac{\beta}{2}|\mathbf{h}_0|^2 \Big)dx \\
&\hspace{30mm}+ \int_{\Omega_y}\Big(\frac{1}{2}|\psi_{0y}|^2+\frac{1}{4}|\psi_0|^4+\alpha g(v_0)h(|\psi_0|^2)\Big)dy \leq C_0,
\end{align*}
where $C_0>0$ is independent of $\varepsilon$. Then, there exists $C=C(C_0)>0$, independent of $\varepsilon$ such that 
\begin{flalign}
&\int_\Omega\Big(\rho\Big(e(\rho,\theta)+ \frac{1}{2}u^2+\frac{1}{2}|\mathbf{w}|^2 \Big) +\frac{\beta}{2}|\mathbf{h}|^2 \Big)dx \nonumber\\
&\hspace{15mm}+ \int_{\Omega_y}\Big(\frac{1}{2}|\psi_y|^2+\frac{1}{4}|\psi|^4+\alpha g(v)h(|\psi|^2)\Big)dy \leq C.&&\label{unifE1}
\end{flalign}
Also,
\begin{equation}
\int_\Omega \rho(\theta-1-\log \theta)dx+ \int_0^t\int_\Omega \Big( \frac{\kappa \theta_y^2}{\theta^2} + \varepsilon u_x^2 + \mu |\mathbf{w}_x|^2+\nu|\mathbf{h}_x|^2 \Big)dx\hspace{.5mm}ds  \leq C.\label{entrunif}
\end{equation}
\end{lemma}

\begin{proof}
First, \eqref{unifE1} follows directly from the energy identity \eqref{difE1}.

Second, from equation \eqref{E1rho} we have that 
\begin{equation}
\int_\Omega \rho dx = \int_\Omega \rho_0 dx, \label{masscons}
\end{equation}

Now, using relations \eqref{entropydef} we see that the entropy $s$ satisfies the following equation
\begin{align}
&(\rho s)_t + (\rho u s)_x-\left(\frac{\kappa \theta_x}{\theta} \right)_x=\frac{\kappa \theta_x^2}{\theta^2}+ \frac{\varepsilon u_x^2}{\theta}+\frac{\mu |\mathbf{w}_x|^2}{\theta}+\frac{\nu |\mathbf{h}|_x^2}{\theta}\label{E1s}
\end{align}
From the definition of $s$ and using \eqref{defe} and \eqref{Q} we have that
\begin{align*}
-\int_\Omega \rho s dx &\geq C^{-1}\int_\Omega \rho(\theta-1-\log \theta)dx -C-C\int_\Omega \rho e(\rho,\theta)dx\\
&\geq C^{-1}\int_\Omega \rho(\theta-1-\log \theta)dx -C.
\end{align*}

Then, integrating equation \eqref{E1s} over $\Omega\times(0,t)$ we get
\begin{align*}
&\int_\Omega \rho(\theta-1-\log \theta)dx + \int_0^t\int_\Omega \left( \frac{\kappa \theta_y^2}{\theta^2} + \frac{\varepsilon u_x^2}{\theta}+\frac{\mu |\mathbf{w}_x|^2}{\theta}+\frac{\nu |\mathbf{h}|_x^2}{\theta} \right)dx\hspace{.5mm}ds \leq C.
\end{align*}

Next, integrating equation \eqref{E1Q} (remember that we introduced the coefficient $\delta$ multiplying the thermal part of the pressure) and using \eqref{unifE1} together with \eqref{deltape+ptet}, \eqref{ptet} and our assumption that $\delta=o(\varepsilon)$ we have
\begin{align*}
&\int_0^t\int_\Omega (\varepsilon u_x^2+|\mathbf{w}_x|^2+\nu|\mathbf{h}_x|^2)dx\hspace{.5mm}ds\\
&\hspace{20mm}\leq C+C\int_0^t M_\theta(s)^2\int_\Omega(1+\rho^\gamma)dx\hspace{.5mm}ds+\frac{\varepsilon}{2}\int_0^t\int_\Omega u_x^2dx\hspace{.5mm}ds.
\end{align*}

Here, $M_\theta(t)=\max_{x\in\Omega}\theta(x,t)$. Now, according to \eqref{Q} and using \eqref{unifE1} we have that $\int_\Omega \rho \theta dx\leq C$. Also, we see that for any $t\in [0,T]$ there is a point $b=b(t)\in\Omega$ such that $\theta(b(t),t)=\left(\int_\Omega \rho dx\right)^{-1}\int_\Omega \rho \theta dx\leq C$. Thus, similarly as in \eqref{Mtetint}, using \eqref{K} we have
\begin{equation}
\int_0^T M_\theta(s)^2ds\leq C+C\int_0^T\int_\Omega \frac{\kappa \theta_x^2}{\theta^2}\leq C.\label{EMtetint}
\end{equation}
Also notice that by \eqref{defe}
\[
\int_\Omega \rho^\gamma dx \leq C\int_\Omega \rho e dx\hspace{.5mm}ds\leq C,
\]
and hence,
\[
\int_0^t\int_\Omega (\varepsilon u_x^2 + \mu|\mathbf{w}_x|^2+\nu|\mathbf{h}_x|^2)dx\hspace{.5mm}ds\leq C.
\]
\end{proof}

We now establish an estimate for the spatial derivative of the density.

\begin{lemma}\label{uniform2}
Let $\alpha=o(\varepsilon^{1/2})$, $\delta=o(\varepsilon)$ and $\beta=o(\varepsilon)$. Assume that $\rho_0$, $u_0$ and $\mathbf{h}_0$ satisfy
\[
\varepsilon^2\int_\Omega \frac{\rho_{0x}^2}{\rho_0^3}dx + \varepsilon\beta^2\int_\Omega\frac{|\mathbf{h}_0|^2}{\rho_0}dx\leq C_0,
\]
and
\[
C_0^{-1}\leq \int_\Omega \rho_0 dx\leq C_0.
\]
where $C_0$ is independent of $\varepsilon$. Then, there exists $C=C(C_0)$ such that
\begin{flalign}
&\varepsilon^2\int_\Omega \frac{\rho_{x}^2}{\rho^3}dx + \varepsilon\beta^2\int_\Omega\frac{|\mathbf{h}|^2}{\rho}dx\nonumber\\
&\hspace{15mm}+\varepsilon \int_0^t\int_\Omega \rho_x^2\rho^{\gamma-3}dx\hspace{.5mm}ds+\varepsilon \beta \int_0^t\int_\Omega \frac{|\mathbf{h}_x|^2}{\rho}dx\hspace{.5mm}ds\leq C.&&\label{unifeq2}
\end{flalign}
\end{lemma}

\begin{proof}
As in \cite{CP} we deduce the following equation for $v(x,t)=1/\rho(x,t)$:

\begin{equation}
(\rho v_x^2)_t+(\rho u v_x^2)_x=2v_x u_{xx}.\label{rhoxrho3}
\end{equation}
Using equation \eqref{E1u} we have
\begin{align}
2 v_x u_{xx}&=\frac{2}{\varepsilon}v_x(p_x + (\rho u)_t +(\rho u^2)_x) + \frac{\beta}{\varepsilon}v_x(|\mathbf{h}|^2)_x-2\frac{\alpha}{\varepsilon}v_x(g'(v)h(|\psi|^2))_x\nonumber\\
  &=\frac{2}{\varepsilon}v_x p_x + \frac{2}{\varepsilon}((\rho u v_x)_t +[\rho u(uv_x)_x-\rho u(v u_x)_x+v_x(\rho u^2)_x])\nonumber\\
  &\hspace{5mm}+\frac{2}{\varepsilon}v_x(|\mathbf{h}|^2)_x-2\frac{\alpha}{\varepsilon}v_x(g'(v)h(|\psi\circ\mathbf{Y}|^2))_x.
\end{align}

Denoting by $J$ the expression in square brackets, by integration by parts, we have
\begin{align*}
\int_\Omega J dx&=\int_\Omega (v u_x (\rho u)_x-uv_x(\rho u)_x + v_x(u(\rho u)_x+\rho u u_x))dx\\
&=\int_\Omega (v u_x(\rho u)_x + \rho u v_x u_x)dx = \int_\Omega u_x^2 dx.
\end{align*}

Next, bearing in mind our assumption \eqref{deltape+ptet} we see that
\[
v_x p_x = -a\gamma \rho^{\gamma-3}\rho_x^2 - \delta\frac{\rho_x}{\rho^2}\theta_xp_\theta(\rho)-\delta\frac{\rho_x^2}{\rho^2}\theta p_\theta'(\rho).
\]

In order to deal with the term $v_x(|\mathbf{h}|^2)_x$ we first rewrite \eqref{E1rho} as
\[
v_t+v_x u =v u_x.
\]

Multiply this equation by $\beta |\mathbf{h}|^2$ to obtain
\[
\beta v_t |\mathbf{h}|^2 + \beta v_x u |\mathbf{h}|^2-\beta v u_x|\mathbf{h}|^2=0.
\]

Now, multiply \eqref{E1h} by $2v\mathbf{h}$ and add the resulting equation to the above to obtain
\[
\beta (v|\mathbf{h}|^2)_t + 2\nu v |\mathbf{h}_x|^2+\beta(v u |\mathbf{h}|^2)_x - (2\nu v \mathbf{h}\cdot\mathbf{h}_x)_x + 2\beta v\mathbf{h}\cdot \mathbf{w}_x=-\nu v_x(|\mathbf{h}|^2)_x.
\]

In this way,
\[
v_x (|\mathbf{h}|^2)_x=-\frac{\beta}{\nu} (v|\mathbf{h}|^2)_t - 2 v |\mathbf{h}_x|^2+\frac{\beta}{\nu}(v u |\mathbf{h}|^2)_x + (2 v \mathbf{h}\cdot\mathbf{h}_x)_x - \frac{2\beta}{\nu} v\mathbf{h}\cdot \mathbf{w}_x.
\]

Gathering this information in \eqref{rhoxrho3}, multiplying by $\varepsilon^2$ and integrating over $\Omega\times(0,t)$ we get
\begin{flalign}
&\varepsilon^2\int_\Omega \frac{\rho_x^2}{\rho^3}dx &\nonumber\\
&=\varepsilon^2\int_\Omega\frac{\rho_{0,x}^2}{\rho_0^3}dx - 2a\gamma\varepsilon\int_0^t\int_\Omega \rho^{\gamma-3}\rho_x^2dx\hspace{.5mm}ds &\nonumber\\
&\hspace{5mm}- 2\varepsilon \delta\int_0^t\int_\Omega\left( \frac{\rho_x}{\rho^2}\theta_x p_\theta(\rho)+\frac{\rho_x^2}{\rho^2}\theta p_\theta'(\rho) \right)dx\hspace{.5mm}ds -2\varepsilon\int_\Omega\frac{\rho_x}{\rho}u dx &\nonumber\\
&\hspace{5mm}+ 2\varepsilon \int_\Omega\frac{\rho_{0x}}{\rho_0}u_0dx + 2\varepsilon\int_0^t\int_\Omega u_x^2 dx-\frac{\varepsilon \beta^2}{\nu}\int_\Omega\frac{1}{\rho}|\mathbf{h}|^2dx&\nonumber\\
&\hspace{5mm}+\frac{\varepsilon \beta^2}{\nu}\int_\Omega\frac{1}{\rho_0}|\mathbf{h}_0|^2dx - 2\varepsilon\beta\int_0^t\int_\Omega \frac{1}{\rho}|\mathbf{h}_x|^2dx\hspace{.5mm}ds + \frac{2\varepsilon\beta^2}{\nu}\int_0^t\int_\Omega \frac{1}{\rho}\mathbf{b}\cdot\mathbf{w}_x dx\hspace{.5mm}&\nonumber\\
&\hspace{5mm}+2\alpha\varepsilon\int_0^t\int_\Omega \frac{\rho_x}{\rho^2}(g'(1/\rho)h(|\psi\circ \mathbf{Y}|^2))_xdx\hspace{.5mm}ds.&\label{rhorho}
\end{flalign}

Concerning the third integral on the right hand side, by virtue of \eqref{ptet}, we have that
\begin{align*}
&-2\varepsilon \delta\int_0^t\int_\Omega\left( \frac{\rho_x}{\rho^2}\theta_x p_\theta(\rho)+\frac{\rho_x^2}{\rho^2}\theta p_\theta'(\rho) \right)dx\hspace{.5mm}ds \\
&\hspace{30mm}\leq 2C\varepsilon \delta\int_0^t\int_\Omega \frac{|\rho_x|}{\rho^2}|\theta_x|(1+\rho^{\gamma/2})dx\hspace{.5mm}ds.
\end{align*}

Observe that since $\int_\Omega \rho dx = \int_\Omega \rho_0 dx$, then for each $t\in(0,T)$ there is a point $b(t)\in\Omega$ such that $\rho(b(t),t)=\int_\Omega \rho_0 dx\geq C_0^{-1}$. Therefore,
\[
\max_{z\in\Omega}\frac{1}{\sqrt{\rho(z,t)}}\leq \sqrt{C_0} + \int_\Omega\left| \left(\frac{1}{\sqrt{\rho}} \right)_x\right|dx\leq C+C\left( \int_\Omega \frac{\rho_x^2}{\rho^3}dx \right)^{1/2}.
\]

Thus, taking \eqref{K} into consideration we see that
\begin{align*}
&2\varepsilon \delta\int_0^t\int_\Omega \frac{|\rho_x|}{\rho^2}|\theta_x|(1+\rho^{\gamma/2})dx\hspace{.5mm}ds\\
&\leq C\varepsilon\int_0^t\max_{x\in\Omega}\frac{1}{\rho^{1/2}}\int_\Omega\left( \delta\frac{|\rho_x|}{\rho^{3/2}}|\theta_x| + \delta|\theta_x|\hspace{.5mm}|\rho_x|\rho^{(\gamma-3)/2} \right)dx\hspace{.5mm}ds\\
&\leq C\varepsilon \int_0^t\left(C+C\left( \int_\Omega \frac{\rho_x^2}{\rho^3}dx \right)^{1/2}  \right)\left( \int_\Omega \frac{\kappa \theta_x^2}{\theta^2}dx \right)^{1/2}\times\\
&\hspace{40mm}\times\left[ \left(\int_\Omega \varepsilon^2 \frac{\rho_x^2}{\rho^3}dx \right)^{1/2}+\left(\int_\Omega \varepsilon^2\rho_x^2 \rho^{\gamma-3} \right)^{1/2} \right]ds\\
&\leq \frac{a\gamma\varepsilon}{4}\int_0^t\int_\Omega\rho_x^2\rho^{\gamma-3}dx\hspace{.5mm}ds + C\int_0^t\left( 1+\int_\Omega\frac{\kappa\theta_x^2}{\theta^2}dx \right)\left(1+\int_\Omega\varepsilon^2\frac{\rho_x^2}{\rho^3} dx\right)ds.
\end{align*}

We already know from Lemma \ref{uniform1} that 
\[
\varepsilon\int_0^t\int_\Omega u_x^2 dx\hspace{.5mm}ds\leq C.
\]

Concerning the fourth integral on the right hand side
\[
2\varepsilon\int_\Omega\frac{\rho_x}{\rho}u dx\leq \frac{\varepsilon^2}{4}\int_\Omega \frac{\rho_x^2}{\rho^3}dx + C\int_\Omega \rho u^2dx\leq \frac{\varepsilon^2}{4}\int_\Omega \frac{\rho_x^2}{\rho^3}dx+C.
\]

We continue with (recall that $\beta=o(\varepsilon)$)
\begin{align*}
&\frac{2\varepsilon\beta^2}{\nu}\int_0^t\int_\Omega\frac{1}{\rho}\mathbf{h}\cdot\mathbf{w}_x dx\hspace{.5mm}ds\\
&\leq \frac{2\varepsilon\beta^2}{\nu}\int_0^t\max_{x\in\Omega}\frac{1}{\rho^{1/2}}\left( \int_\Omega\frac{1}{\rho}|\mathbf{h}|^2dx \right)^{1/2}\left( \int_\Omega|\mathbf{w}_x|^2 dx  \right)^{1/2}ds\\
&\leq C\frac{\varepsilon\beta^2}{\nu}\int_0^t\left( 1+\left(\int_\Omega\frac{\rho_x^2}{\rho^3}dx \right)^{1/2} \right)\left( \int_\Omega \frac{1}{\rho}|\mathbf{h}|^2dx\right)^{1/2}\left( \int_\Omega |\mathbf{w}_x|^2 dx  \right)^{1/2}ds\\
&\leq C\int_0^t\left( \varepsilon+\left(\int_\Omega\varepsilon^2\frac{\rho_x^2}{\rho^3}dx \right)^{1/2} \right)\left( \int_\Omega \frac{\varepsilon\beta^2}{\nu}\frac{1}{\rho}|\mathbf{h}|^2dx\right)^{1/2}\left( \int_\Omega \frac{\beta^2}{\varepsilon\nu} |\mathbf{w}_x|^2 dx  \right)^{1/2}ds\\
&\leq C\int_0^t \left(1+\int_\Omega \mu |\mathbf{w}_x|^2 dx\right)\left(1+\int_\Omega\varepsilon^2\frac{\rho_x^2}{\rho^3}dx+ \int_\Omega \varepsilon\beta^2\frac{1}{\rho}|\mathbf{h}|^2dx \right)ds.
\end{align*}

Finally, recalling \eqref{Lag3} we see that $(\psi\circ\mathbf{Y})_x=\rho\psi_y$. We also know that the Jacobian of the Lagrangian coordinate change is equal to $\rho$. Therefore, using \eqref{gh} and Lemma \ref{uniform1} we see that
\begin{align*}
&2\alpha\varepsilon\int_0^t\int_\Omega\frac{\rho_x}{\rho^2}(g'(1/\rho)h(|\psi\circ\mathbf{Y}|^2)))_x dx\hspace{.5mm ds}\\
&\hspace{30mm}\leq \frac{a\gamma\varepsilon}{8}\int_0^t\int_\Omega \rho_x^2\rho^{\gamma-3}dx\hspace{.5mm}ds + C\int_0^t\int_\Omega|\psi_x|^2dx\hspace{.5mm}ds\\
&\hspace{30mm}\leq \frac{a\gamma\varepsilon}{8}\int_0^t\int_\Omega \rho_x^2\rho^{\gamma-3}dx\hspace{.5mm}ds +C.
\end{align*}

Putting all of these estimates together with \eqref{rhorho} we deduce the inequality
\begin{align*}
&\varepsilon^2\int_\Omega \frac{\rho_{x}^2}{\rho^3}dx + \varepsilon\beta^2\int_\Omega\frac{|\mathbf{h}|^2}{\rho}dx+\varepsilon \int_0^t\int_\Omega \rho_x^2\rho^{\gamma-3}dx\hspace{.5mm}ds+\varepsilon \beta \int_0^t\int_\Omega \frac{|\mathbf{h}_x|^2}{\rho}dx\hspace{.5mm}ds\\
&\leq C+C\int_0^t\left[1+\int_\Omega\left( \frac{\kappa\theta_x^2}{\theta^2}+\mu|\mathbf{w}_x|^2 \right)dx  \right]\times\\
&\hspace{40mm}\times\left( 1+\varepsilon^2\int_\Omega\frac{\rho_x^2}{\rho^3}dx+\varepsilon\beta^2\int_\Omega\frac{1}{\rho}|\mathbf{h}|^2dx \right) ds,
\end{align*}
with $C>0$ independent of $\varepsilon$. And since,
\[
\int_0^t\left[1+\int_\Omega\left( \frac{\kappa\theta_x^2}{\theta^2}+\mu|\mathbf{w}_x|^2 \right)dx  \right]ds\leq C,
\]
Gronwall's inequality yields \eqref{unifeq2}.
\end{proof}

We now deduce some higher integrability estimates for the density.

\begin{lemma}\label{uniform3}
Let
\[
\int_\Omega \rho_0 e(\rho_0,\theta_0)dx+\int_\Omega\rho_0 u_0^2dx\leq C_0
\]
where $C_0$ is independent of $\varepsilon$. Then, there is a constant $C=C(C_0)$, independent of $\varepsilon$ such that
\begin{equation}
\int_0^t\int_\Omega (\rho^{\gamma+1}+\delta\rho\theta p_\theta(\rho)+\beta\rho|\mathbf{h}|^2)dx\hspace{.5mm}ds\leq C.
\end{equation}
\end{lemma}

Let us point out that according to the growth conditions \eqref{defe}, Lemma \ref{uniform1} yields only uniform boundedness of $\rho$ in the space $L^\infty(0,T;L^\gamma(\Omega))$. Let us carry on the proof.

\begin{proof}
Let $b\in\{ 0,1\}$ (recall that we are assuming that $\Omega=(0,1)$ without loss of generality) and let $\sigma(x)$ be a smooth function such that
\begin{equation}
\sigma(b)=0\hspace{3mm}\text{ and }\hspace{3mm}0\leq \sigma\leq 1.\label{sigm}
\end{equation}

Multiplying equation \eqref{E1u} by $\sigma$ and integrating from $b$ to $x$ (with respect to the space variable) we have
\begin{align*}
p\sigma + \frac{\beta}{2}|\mathbf{h}|^2=&-\rho u^2\sigma + \varepsilon u_x\sigma+\alpha g'(1/\rho)h(|\psi|^2)\sigma - \left( \int_b^x \rho u \sigma d\xi \right)_t \\
&+ \int_b^x\left[\left(\rho u^2 + p + \frac{\beta}{2}|\mathbf{h}|^2-\alpha g'(1/\rho)h(|\psi|^2)\right)\sigma_x - \varepsilon u_x\sigma_x \right]d\xi.
\end{align*}

Multiply this identity by $\rho\sigma$ and use \eqref{E1rho} to obtain
\begin{flalign*}
&\rho p\sigma^2+\frac{\beta}{2}\rho|\mathbf{h}|^2\sigma^2&\\
&\hspace{10mm}=\varepsilon\rho u_x\sigma^2 + \alpha g'(1/\rho)h(|\psi|^2)\rho\sigma^2-\left( \rho\sigma\int_b^x\rho u \sigma d\xi \right)_t &\\
&\hspace{15mm}- \left(\rho u \sigma\int_b^x \rho u \sigma d\xi\right)_x+\rho u\sigma_x\int_0^x\rho u\sigma d\xi  &\\
&\hspace{15mm}+\rho \sigma\int_b^x\left[\left(\rho u^2 + p + \frac{\beta}{2}|\mathbf{h}|^2-\alpha g'(1/\rho)h(|\psi|^2)\right)\sigma_x - \varepsilon u_x\sigma_x \right]d\xi.&
\end{flalign*}

Integrating over $\Omega\times(0,t)$ we have
\begin{align}
&\int_0^t\int_\Omega \rho p \sigma^2dx\hspace{.5mm}+\frac{\beta}{2}\int_0^t\int_\Omega \rho|\mathbf{h}|^2dx\hspace{.5mm}ds\nonumber\\
&=\alpha \int_0^t\int_\Omega g'(1/\rho)h(|\psi|^2)\rho\sigma^2dx\hspace{.5mm}ds + \varepsilon\int_0^t\int_\Omega \rho u_x \sigma^2dx\hspace{.5mm}ds\nonumber\\
&\hspace{7mm}+\int_\Omega\rho\sigma\left( \int_b^x\rho u\sigma dxi\right)dx + \int_\Omega\rho_0\sigma\left( \int_b^x\rho_0 u_0\sigma dxi\right)dx + r_1(t),\label{uniform3.1}
\end{align}
where
\begin{align*}
&r_1(t)=\int_0^t\int_\Omega \rho u\sigma_x\left(\int_b^x \rho u\sigma d\xi\right)dx\hspace{.5mm}ds\\
  &+\int_0^t\int_\Omega \rho \sigma\int_b^x\left[\left(\rho u^2 + p + \frac{\beta}{2}|\mathbf{h}|^2-\alpha g'(1/\rho)h(|\psi|^2)\right)\sigma_x - \varepsilon u_x\sigma_x \right]d\xi dx\hspace{.5mm}ds.
\end{align*}

Note that,
\begin{align*}
\varepsilon\int_0^t\int_\Omega \rho u_x \sigma^2dx\hspace{.5mm}ds&\leq \frac{\varepsilon}{\delta_1}\int_0^t\int_\Omega u_x^2dx\hspace{.5mm}ds + \delta_1\int_0^t\int_\Omega\rho^2\sigma^2dx\hspace{.5mm}ds\\
&\leq C_{\delta_1} + C\delta_1\int_0^t\int_\Omega\rho p\sigma^2dx\hspace{.5mm}ds.
\end{align*}

On the other hand, by virtue of the estimates in Lemma \ref{uniform1}, all the other terms on the right hand side of \eqref{uniform3.1} are bounded. Thus, choosing $\delta_1>0$ small enough, we get
\[
\int_0^t\int_\Omega \left(\rho p +\frac{\beta}{2}\rho|\mathbf{h}|^2\right)\sigma^2 dx\hspace{.5mm}ds\leq C,
\]
and using the hypotheses \eqref{pe+ptet} on $p$ 
\[
\int_0^t\int_\Omega \left(\rho^{\gamma+1}+\delta\theta\rho p_\theta(\rho) +\frac{\beta}{2}\rho|\mathbf{h}|^2\right)\sigma^2 dx\hspace{.5mm}ds\leq C,
\]
and this holds for any $\sigma$ that satisfies \eqref{sigm} (of course, the constant $C$ in this last inequality depends on $\sigma$). Choosing $\sigma_1(x)=x$ and $\sigma_2(x)=1-x$ (again, $\Omega=(0,1)$) we get
\[
\int_0^t\int_\Omega \left(\rho^{\gamma+1}+\delta\theta\rho p_\theta(\rho) +\frac{\beta}{2}\rho|\mathbf{h}|^2\right)(x^2+(1-x)^2) dx\hspace{.5mm}ds\leq C.
\]

Since $\min_{x\in(0,1)}(x^2+(1-x)^2)=1/2$ we conclude that
\[
\int_0^t\int_\Omega \left(\rho^{\gamma+1}+\delta\theta\rho p_\theta(\rho) +\frac{\beta}{2}\rho|\mathbf{h}|^2\right)dx\hspace{.5mm}ds\leq C.
\]
\end{proof}

We now deduce a higher order integrability estimate for the velocity.

\begin{lemma}\label{uniform4}
Let $\alpha=o(\varepsilon^{1/2})$ and $\beta=o(\varepsilon)$. Assume that $(\rho_0,u_0,\mathbf{h}_0,\theta_0)$ satisfy
\[
\int_\Omega \rho_0 dx\geq C^{-1},\hspace{10mm}\varepsilon^2\int_\Omega \frac{\rho_{0,x}^2}{\rho_0^3}dx + \varepsilon\beta^2\int_\Omega \frac{1}{\rho_0}|\mathbf{h}_0|^2dx\leq C_0,
\]
and
\[
\int_\Omega \rho_0(e(\rho_0,\theta_0)+u_0^2)dx\leq C_0
\]
where, $C_0$ is a constant independent of $\varepsilon$. Then
\begin{equation}
\int_0^t\int_\Omega(\rho |u|^3+\rho^{\gamma+\vartheta})dx\hspace{.5mm}ds\leq C,\label{highint2}
\end{equation}
where, $\vartheta=\frac{\gamma-1}{2}$ and $C>$ is a constant independent of $\varepsilon$.
\end{lemma}

\begin{proof}
Let $\zeta_\#(z)=\frac{1}{2}z|z|$. Then, the corresponding weak entropy pair $(\eta^\#,q^\#)$ $:=(\eta^{\zeta_\#},q^{\zeta_\#})$ satisfies
\begin{align}
&|\eta^\#(\rho,m)|\leq C(\rho|u|^2+\rho^\gamma), &C^{-1}(\rho|u|^3+\rho^{\gamma+\vartheta})\leq q^\#(\rho,m)\leq C(\rho|u|^3+\rho^{\gamma+\vartheta}),\label{2.14}\\
&|\eta_m^\#(\rho,m)|\leq C(|u|+\rho^\vartheta), &|\eta_{mm}^\#(\rho,m)\leq C\rho^{-1}.\label{2.15}
\end{align}
and, regarding $\eta_m^\#$ in the coordinates $(\rho,u)$
\begin{align}
&|\eta_{mu}^\#(\rho,\rho u)|\leq C, &|\eta_{m\rho}^\#(\rho,u)|\leq C\rho^{\vartheta-1},\label{2.16}
\end{align}
for all $\rho\geq 0$ and all $u\in\mathbb{R}$. This is a consequence of the representation formulas \eqref{entropexpl}.

Multiply \eqref{E1rho} by $\eta_\rho^\#$ and \eqref{E1u} by $\eta_u^\#$ and add the resulting equations to obtain
\begin{equation}
\eta^\#(\rho,m)_t + q^\#(\rho,m)_x=\left( -\frac{\beta}{2}|\mathbf{h}|^2 + \alpha g'(1/\rho)h(|\psi|^2) + \varepsilon u_x  \right)_x \eta_m^\#(\rho,m).\label{entrfluxeq}
\end{equation}

Define the function
\begin{equation}
f(x,t):=\big[ q^\#(\rho,m) +\big(\tfrac{\beta}{2}|\mathbf{h}|^2-\alpha g'(1/\rho)h(|\psi\circ Y|^2) -\varepsilon u_x \big)\eta_m^\#(\rho,m) \big](x,t).
\end{equation}

We claim that there is a function $a(t)$ taking values in $\Omega$ such that
\begin{equation}
\int_0^t |f(a(s),s)|ds \leq C,\label{claim}
\end{equation}
for some $C>0$ independent of $\varepsilon$.

Assuming this for now, we integrate \eqref{entrfluxeq} over $(a,x)\times(0,t)$ and get
\begin{align}
&\int_a^x (\eta^\#(\rho,m)-\eta^\#(\rho_0,m_0)) d\xi + \int_0^t q^\#(\rho,m) ds\nonumber\\
&= \int_0^t f(a(s),s)ds+\int_0^t  \left( -\frac{\beta}{2}|\mathbf{h}|^2+\alpha g'(1/\rho)h(|\psi\circ Y|^2)+\varepsilon u_x \right)\eta_m^\# ds \nonumber\\
&\hspace{5mm}-\int_0^t\int_a^x \left( -\frac{\beta}{2}|\mathbf{h}|^2+\alpha g'(1/\rho)h(|\psi|^2)+\varepsilon u_x \right)(\eta_{m\rho}^\#\rho_x + \eta_{mu}^\# u_x )dx\hspace{.5mm}ds.\label{etanum}
\end{align}

First, from \eqref{2.14} we see that
\[
\int_0^t\int_\Omega q^\#(\rho,m)dx\hspace{.5mm}ds \geq C^{-1}\int_0^t\int_\Omega (\rho|u|^3+\rho^{\gamma+\vartheta}).
\]

Second, from Lemma \ref{uniform1} and \eqref{2.14}
\[
\int_0^t\int_\Omega (|\eta^\#(\rho,m)|+|\eta^\#(\rho_0,m_0)|)dx\leq C.
\]

Next, using the fact that $\mathbf{h}|_{x=0}=0$ we see that
\begin{equation}
\beta^{1/2} |\mathbf{h}|^2 \leq 2\left(\beta \int_\Omega |\mathbf{h}|^2\right)^{1/2}\left( \int_\Omega |\mathbf{h}_x|^2\right)^{1/2}\leq C\left( \int_\Omega |\mathbf{h}_x|^2\right)^{1/2}.\label{hL1infty}
\end{equation}

Similarly,
\begin{equation}
\varepsilon^{1/2}|u|\leq \int_\Omega \varepsilon^{1/2}|u_x|dx\leq C\left( \varepsilon\int_\Omega u_x^2dx\right)^{1/2}.\label{uL1infty}
\end{equation}

Using these two observations along with \eqref{2.15} and Lemma \ref{uniform1}
\begin{align*}
&\int_0^t\int_\Omega \left( -\frac{\beta}{2}|\mathbf{h}|^2+\alpha g'(1/\rho)h(|\psi|^2)+\varepsilon u_x \right)\eta_m^\#dx\hspace{.5mm}ds\\
&\hspace{5mm}\leq C+C\int_0^t\int_\Omega (\nu|\mathbf{h}_x|^2+\varepsilon u_x^2)dx\hspace{.5mm}ds + C(\beta+\alpha^2+\varepsilon)\int_0^t\int_\Omega (u_x^2+\rho^{\gamma-1})dx\hspace{.5mm}ds\\
&\hspace{5mm}\leq C.
\end{align*}

Finally, by the same reasoning and using \eqref{2.16} and Lemma \ref{uniform3} we see that
\begin{align*}
&\int_0^t\int_\Omega \left( -\frac{\beta}{2}|\mathbf{h}|^2+\alpha g'(1/\rho)h(|\psi|^2)+\varepsilon u_x \right)(\eta_{m\rho}^\#\rho_x + \eta_{mu}^\# u_x )dx\hspace{.5mm}ds\\
&\hspace{5mm}\leq C+C\int_0^t\int_\Omega \nu|\mathbf{h}_x|^2dx\hspace{.5mm}ds+C(\beta+\alpha^2+\varepsilon)\int_0^t\int_\Omega u_x^2 dx\hspace{.5mm}ds\\
&\hspace{15mm}+C(\beta+\alpha^2+\varepsilon)\int_0^t\int_\Omega \rho_x^2\rho^{\gamma-3} dx\hspace{.5mm}ds\\
&\hspace{5mm}\leq C.
\end{align*}

Taking this information into account, integrating \eqref{etanum} over $\Omega$ and using \eqref{claim} we obtain \eqref{highint2}.

In order to complete the proof we have to prove our claim. For this, fix $k\in\mathbb{N}$ large enough so that $\gamma \geq \max\{ 1+\tfrac{2}{2k-3}, 1+\tfrac{1}{2(k-1)} \}$ and observe that
\[
\rho(x,t)\min_{z\in\Omega}|f(z,t)|^{1/k}\leq \rho(x,t)|f(x,t)|^{1/k}\leq \rho(x,t)\max_{z\in\Omega}|f(z,t)|^{1/k}.
\]

Integrating over $\Omega$ and using \eqref{masscons} we see that for a.e. $t$ there is a point $a=a(t)\in\Omega$ such that
\[
|f(a(t),t)|=\left( \int_\Omega \rho_0 dx \right)^{-k}\left( \int_\Omega\rho(x,t)|f(x,t)|^{1/k} dx \right)^k.
\]

Let us show that 
\begin{equation}
\int_0^t\left( \int_\Omega\rho(x,s)|f(x,s)|^{1/k} dx \right)^kds\leq C.\label{claim1}
\end{equation}

On the one hand, since $k$ was chosen so that $\gamma\geq 1+\tfrac{2}{2k-3}$ (which implies that $\tfrac{1}{2k}\leq \tfrac{\gamma-1}{3\gamma-1}$), we can use \eqref{2.14} and \eqref{unifE1} in order to show that
\begin{align*}
&\int_\Omega \rho |q^\#(\rho,m)|^{1/k}\\
&\leq C\int_\Omega \big( \rho^{(2k-1)/2k}(\rho u^2)^{3/2k}+\rho^{1+(3\gamma-1)/2k} \big)dx\\
&\leq C\left( \int_\Omega\rho^{(2k-1)/(2k-3)}dx \right)^{(2k-3)/2k}\left( \int_\Omega \rho u^2 dx \right)^{3/2k} + C\int_\Omega \rho^{1+(3\gamma-1)/2k} dx\\
&\leq C\left(1+ \int_\Omega\rho^\gamma dx \right)^{(2k-3)/2k}\left( \int_\Omega \rho u^2 dx \right)^{3/2k} + C\left(1+\int_\Omega \rho^\gamma dx\right)\leq C
\end{align*}

On the other hand, since $\gamma\geq 1+\tfrac{1}{2(k-1)}$ we have that
\begin{align*}
&\int_\Omega \rho |\varepsilon u_x \eta_m^\#|^{1/k}dx\\
&\leq \varepsilon^{1/2k}\int_\Omega \rho^{1-1/2k}(\varepsilon u_x^2)^{1/2k}((\rho u^2)^{1/2k}+\rho^{\gamma/2k})dx\\
&\leq C\left( \int_\Omega \rho^{1+1/2(k-1)}dx \right)^{(k-1)/k}\left(\int_\Omega \varepsilon u_x^2 dx \right)^{1/2k} \left( \int_\Omega (\rho u^2 +\rho^\gamma) dx \right)^{1/2k} \\
&\leq C \left(1+ \int_\Omega \rho^\gamma dx \right)^{(k-1)/k}\left(1+\int_\Omega \varepsilon u_x^2 dx \right)^{1/k}\\
&\leq C\left(1+\int_\Omega \varepsilon u_x^2 dx \right)^{1/k}.
\end{align*}

With this, we conclude that
\[
\int_0^t \left( \int_\Omega \rho |\varepsilon u_x\eta_m^\#|^{1/k}dx \right)^kds \leq C.
\]

Finally, by the same reasoning, we see that
\begin{align*}
&\int_\Omega \rho \left| \left(\frac{\beta}{2}|\mathbf{h}|^2-\alpha g'(1/\rho)h(|\psi\circ Y|^2) \right)\eta_m^\#\right|^{1/k}dx\\
&\qquad\leq C\left(1+\nu\int_\Omega |\mathbf{h}_x|^2dx\right)^{1/k}\left( \int_\Omega \rho dx \right)^{(2k-1)/2k}\left( \int_\Omega (\rho u^2+\rho^\gamma)dx \right)^{1/2k}\\
&\qquad \leq C\left(1+\nu\int_\Omega |\mathbf{h}_x|^2dx\right)^{1/k},
\end{align*}
which implies that
\[
\int_0^t\left( \int_\Omega \rho \left| \left(\frac{\beta}{2}|\mathbf{h}|^2-\alpha g'(1/\rho)h(|\psi\circ Y|^2) \right)\eta_m^\#\right|^{1/k}dx\right)^k ds\leq C,
\]
thus proving \eqref{claim1}.
\end{proof}

These last four Lemmas provide the necessary uniform estimates that allow us apply Chen and Perepelitsa's compactness scheme in order to deal with the convergence in the continuity equation \eqref{E1rho} and in the momentum equation \eqref{E1u}. They also suffice to handle the convergence issues in equations \eqref{E1w}, \eqref{E1h} and \eqref{E1Sch} to the extent that was explained in Section \ref{limit}. Yet, in order to address the convergence issues in the thermal energy equation \eqref{E1Q} we need one more estimate that reads as follows.

\begin{lemma}\label{uniform5}
Let $\delta=o(\varepsilon)$. Assume that
\[
C^{-1}\leq \int_\Omega \rho_0 dx\leq C_0,\hspace{10mm}\int_\Omega \rho_0 e(\rho_0,\theta_0)dx\leq C_0
\]
for some $C_0>0$ independent of $\varepsilon$. Then,
\begin{equation}
\int_0^t\int_\Omega (\theta^{q+1}+|(\theta^{q/2})_x|^2) dx\hspace{.5mm} ds \leq C,\label{uniftetq}
\end{equation}
where, $q$ is as in \eqref{K} and $C>0$ is constant independent of $\varepsilon$.
\end{lemma}

\begin{proof}
Let us define $\mathcal{K}$ as in Section \ref{limit} as
\[
\mathcal{K}(\theta):=\int_0^\theta \kappa (z)dz.
\]

Then, from \eqref{K} we have that
\[
C^{-1}(1+\theta^{q+1})\leq \mathcal{K}(\theta)\leq C(1+\theta^{q+1}).
\]

Also, note that equation \eqref{E1Q} can be rewritten as
\begin{equation}
\mathcal{K}_{xx}=(\rho Q(\theta))_t+(\rho u Q(\theta))_x+\delta\theta p_\theta(\rho) u_x-\varepsilon u_x^2 - \mu |w_x|^2 - \nu |\mathbf{h}_x|^2.
\end{equation}

Realizing that $(\mathcal{K}_x)|_{x=0}=0$, we integrate this equation over $(0,x)\times(0,t)$ to obtain
\begin{align*}
\int_0^t \mathcal{K}_x ds &= \int_0^x (\rho Q(\theta)-\rho_0 Q(\theta_0) )d\xi+\int_0^t \rho Q(\theta)u ds + \delta\int_0^t\int_0^x \theta p_\theta u_x d\xi\hspace{.5mm}ds \\
 &\hspace{10mm}-\int_0^t\int_0^x (\varepsilon u_x^2 + \mu |w_x|^2 + \nu |\mathbf{h}_x|^2)d\xi\hspace{.5mm}ds.
\end{align*}

Let us choose some $b=b(t)\in\Omega$ such that $\theta(b(t),t)=\left( \int_\Omega \rho_0 dx \right)^{-1}\left( \int_\Omega \rho \theta dx\right)$ $\leq C$. Then, integrating the above equality from $b(t)$ to $x$ (with respect to the space variable) and using Lemma \ref{uniform1} we obtain
\[
\int_0^t \mathcal{K} ds \leq C.
\]

In particular,
\[
\int_0^t\int_\Omega \theta^{q+1}dx\hspace{.5mm}ds\leq C.
\]

In order to conclude, we observe that 
\[
\int_0^t\int_\Omega|\nabla(\theta^{q/2})_x|^2 dx\hspace{.5mm} ds \leq C+C\int_0^t\int_\Omega\frac{\kappa \theta_x^2}{\theta^2} dx\hspace{.5mm} ds\leq C.
\]
\end{proof}

With these estimates at hand we are ready to pass to the limit as $\varepsilon\to 0$

\section{Limit process}\label{process}

\subsection{Limit in the continuity and momentum equations}

Let $(\rho^\varepsilon,u^\varepsilon,\mathbf{w}^\varepsilon,\mathbf{h}^\varepsilon,\theta^\varepsilon,\psi^\varepsilon)$ be the unique global solution of \eqref{E1rho}-\eqref{E1Sch}. Let us consider the sequence $(\rho^\varepsilon,u^\varepsilon)$. The first step is to apply the Young measures theorem to the sequence $(\rho^\varepsilon,u^\varepsilon)$. Note that this sequence takes values in the set $\mathcal{H}:=\{ (\rho,u)\in\mathbb{R}^2:\rho > 0 \}$. Nonetheless, we do not have any uniform estimates that guarantee that this sequence takes values on a fixed compact of $\mathcal{H}$.

With this in mind, and following \cite{CP} (cf. \cite{LW}), we consider a compactification $\mathbb{H}$ of $\mathcal{H}$ such that the space $C(\mathbb{H})$ is isometrically isomorphic to the space of continuous functions $\phi\in C(\overline{\mathcal{H}})$ satisfying that $\phi(\rho,u)$ is constant on the vacuum $\{ \rho=0\}$ and that the map $(\rho,u)\to \lim_{s\to \infty}\phi(s\rho,su)$ belong to $C(\mathbb{S})\cap \mathcal{H}$, where $\mathbb{S}\subseteq \mathbb{R}^2$ is the unit circle. Of course, $\mathcal{H}$ is naturally embedded in $\mathbb{H}$ (note that the vacuum line $V=\{ \rho=0\}$ is identified to a single point in $\mathbb{H}$).

By the Young measures theorem (\cite[Theorem 2.4]{AM}, also \cite{B}) there exists a subsequence, still denoted $(\rho^\varepsilon,u^\varepsilon)$ and a weakly measurable mapping from $\Omega\times [0,\infty)$ to $\text{Prob}(\mathbb{H})$, the space of probability measures in $\mathbb{H}$, $(x,t)\to\nu_{x,t}$ such that for all $\phi\in C(\mathbb{H})$
\[
\phi(\rho^\varepsilon,u^\varepsilon)\rightharpoonup \int_{\mathbb{H}}\phi(\rho,u)d\nu_{x,t},\hspace{5mm}\text{ weakly}-*\text{ in }L^\infty(\Omega\times [0,\infty)).
\]

In order to prove strong convergence of the sequence $(\rho^\varepsilon,\rho^\varepsilon  u^\varepsilon)$, and consequently that the limit $(\rho,\rho u)$ is a weak solution of the isentropic Euler equations \eqref{Eulerrho}, \eqref{Euleru},  it suffices to show that the Young measures are reduced to delta masses. That is, we have to show that
\begin{equation}
\nu_{x,t}=\delta_{(\rho(x,t),\rho(x,t) u(x,t))}\text{ for a.e. }(x,t).\label{nueqdelta}
\end{equation}

This is achieved by showing that the parametrized Young measure satisfies the Tartar-Murat commutator relation, from which the arguments in \cite{CP,C,DC} (cf. \cite{DP,LPS,LPT}) imply \eqref{nueqdelta}.

To this end, just as in \cite{CP}, we can show the following.
\begin{proposition}
The following statements hold: 
\begin{itemize} 
\item[(i)]
\begin{equation}
\int_\mathcal{H}(\rho^{\gamma+1}+\rho|u|^3)d\nu_{x,t}\in L^1_{loc}(\Omega\times[0,\infty)).\label{rhoulim}
\end{equation}
\item[(ii)]Let $\phi(\rho,u)$ be a function such that
  \begin{enumerate}
  \item[(a)] $\phi$ is continuous on $\overline{\mathcal{H}}$ and zero on $\partial \mathcal{H}$ (in the vacuum);
  \item[(b)] $\text{supp}\phi\subseteq\{ (\rho,u):\rho^\vartheta+u\geq -a, u-\rho^\vartheta\leq a \}$ for some constant $a>0$;
  \item[(c)] $|\phi(\rho,u)|\leq \rho^{\beta(\gamma+1)}$ for all $(\rho,u)$ with large $\rho$ and some $\beta\in(0,1)$. 
\end{enumerate}  
  Then, $\phi$ is $\nu_{x,t}$-integrable and
  \begin{equation}
  \phi(\rho^\varepsilon,u^\varepsilon)\rightharpoonup\int_\mathcal{H}\phi d\nu_{x,t} \hspace{5mm} \text{ in }L^1_{loc}(\Omega\times[0,\infty)).\label{testampl}
  \end{equation}
\item[(iii)] For $\nu_{x,t}$ viewed as an element of $(C(\mathbb{H}))^*$
\[
\nu_{x,t}[\mathbb{H} \setminus (\mathcal{H}\cup V)]=0,
\]
meaning that $\nu_{x,t}$ is concentrated at $\mathcal{H}$ and or the vacuum $V=\{ \rho=0\}$.
\end{itemize}
\end{proposition}

In view of this proposition, and by Lemma \ref{boundsetaq} the entropy pairs are $\nu_{x,t}$-integrable and we can use them as test functions. Once we have Lemmas \ref{uniform3} and \ref{uniform4} the proof follows exactly as the analogue in \cite{CP} and we omit the details.

The next step to take is to apply the Div-Curl Lemma in order to prove the commutator relations for the Young measures.  For this, we need the following result.
\begin{proposition}\label{etaqcomp}
Let $\zeta$ be any compactly supported $C^2$ function and let $(\eta^\zeta,q^\zeta)$ be the corresponding entropy pair given by \eqref{entropexpl}. Then the {\sl entropy dissipation measures}
\[
\eta^\zeta(\rho^\varepsilon,\rho^\varepsilon u^\varepsilon)_t + q^\zeta(\rho^\varepsilon,\rho^\varepsilon u^\varepsilon)_x
\]
belong to a compact of $H_{loc}^{-1}(\Omega\times[0,\infty))$.
\end{proposition}

\begin{proof}
Multiplying \eqref{E1rho} by $\eta_\rho^\zeta$ and \eqref{E1u} by $\eta_m^\zeta$ and adding the resulting equations we obtain
\begin{align}
&\eta^\zeta(\rho^\varepsilon,m^\varepsilon)_t+q^\zeta(\rho^\varepsilon,m^\varepsilon)_x\nonumber\\
&= \varepsilon (\eta_m^\zeta (\rho^\varepsilon,\rho^\varepsilon u^\varepsilon)u_x^\varepsilon)_x - \varepsilon \eta_{mu}(\rho^\varepsilon,\rho^\varepsilon u^\varepsilon)|u_x^\varepsilon|^2-\varepsilon\eta_{m\rho}^\zeta(\rho^\varepsilon,\rho^\varepsilon u^\varepsilon)\rho_x^\varepsilon u_x^\varepsilon\nonumber\\
&\hspace{3mm}-(\delta\theta^\varepsilon p_\theta(\rho^\varepsilon)\eta_m^\zeta(\rho^\varepsilon,\rho^\varepsilon u^\varepsilon))_x + \delta \theta^\varepsilon p_\theta(\rho^\varepsilon) (\eta_{mu}^\zeta(\rho^\varepsilon,\rho^\varepsilon u^\varepsilon) u_x^\varepsilon + \eta_{m\rho}^\zeta(\rho^\varepsilon,\rho^\varepsilon u^\varepsilon)\rho_x^\varepsilon)\nonumber\\
&\hspace{3mm}-\left(\frac{\beta}{2}|\mathbf{h}^\varepsilon|^2 -\alpha g'(1/\rho^\varepsilon)h(|\psi^\varepsilon|^2)\right)_x\eta_m^\zeta(\rho^\varepsilon,\rho^\varepsilon u^\varepsilon).\label{entrdissip}
\end{align}

Using H\"{o}lder inequality and Lemmas \ref{boundsetaq}, \ref{uniform1} and \ref{uniform2} we see that
\begin{align*}
&|| \varepsilon \eta_{mu}(\rho^\varepsilon,\rho^\varepsilon u^\varepsilon)|u_x^\varepsilon|^2-\varepsilon\eta_{m\rho}^\zeta(\rho^\varepsilon,\rho^\varepsilon u^\varepsilon)\rho_x^\varepsilon u_x^\varepsilon ||_{L^1((\Omega\times(0,T))}\\
&\hspace{40mm}\leq C_\zeta ||(\varepsilon^{1/2}u_x^\varepsilon,\varepsilon^{1/2}(\rho^\varepsilon)^{\frac{\gamma-3}{2}}\rho_x^\varepsilon)||_{L^2(\Omega\times(0,T))}\leq C,
\end{align*}

Similarly, using \eqref{EMtetint} and \eqref{ptet}
\begin{align*}
&||\delta \theta^\varepsilon p_\theta(\rho^\varepsilon) (\eta_{mu}^\zeta(\rho^\varepsilon,\rho^\varepsilon u^\varepsilon) u_x^\varepsilon + \eta_{m\rho}^\zeta(\rho^\varepsilon,\rho^\varepsilon u^\varepsilon)\rho_x^\varepsilon)||_{L^1((\Omega\times(0,T))}\leq C,
\end{align*}
and using \eqref{hL1infty}, \eqref{Lag3}, \eqref{gh} and Lemmas \ref{uniform1} and \ref{uniform2}
\[
\left\Vert \left(\frac{\beta}{2}|\mathbf{h}^\varepsilon|^2 -\alpha g'(1/\rho^\varepsilon)h(|\psi^\varepsilon\circ Y|^2)\right)_x\eta_m^\zeta(\rho^\varepsilon,\rho^\varepsilon u^\varepsilon)\right\Vert_{L^1((\Omega\times(0,T))}\leq C
\] 

Also, note that
\[
\Vert \varepsilon \eta_m^\zeta (\rho^\varepsilon,\rho^\varepsilon u^\varepsilon)u_x^\varepsilon\Vert_{L^2((\Omega\times(0,T))}\leq C,
\]
and
\[
\Vert \delta\theta^\varepsilon p_\theta(\rho^\varepsilon)\eta_m^\zeta(\rho^\varepsilon,\rho^\varepsilon u^\varepsilon) \Vert_{L^2((\Omega\times(0,T))}\leq C.
\]

Now, we recall that $L^1(\Omega\times(0,T))$ is compactly embedded into $W^{-1,q}(\Omega\times(0,T))$ for any $1<q<2$, so that from equation \eqref{entrdissip} we can conclude that
\[
\eta^\zeta(\rho^\varepsilon,m^\varepsilon)_t+q^\zeta(\rho^\varepsilon,m^\varepsilon)_x \text{ belong to a compact subset of }W_{loc}^{-1,q_1}(\Omega\times[0,\infty)),
\]
for some $1<q_1<2$.

On the other hand, using the bounds in Lemma \ref{boundsetaq} and the estimates in Lemmas \ref{uniform3} and \ref{uniform4} we have that
\[
\eta^\zeta(\rho^\varepsilon,m^\varepsilon),q^\zeta(\rho^\varepsilon,m^\varepsilon) \text{ are uniformly bounded in }L_{loc}^2(\Omega\times [0,\infty)),
\]
which implies that
\[
\eta^\zeta(\rho^\varepsilon,m^\varepsilon)_t+q^\zeta(\rho^\varepsilon,m^\varepsilon)_x \text{ are uniformly bounded in }W_{loc}^{-1,q_2}(\Omega\times[0,\infty))
\]
where $q_2=\gamma+1>2$, when $\gamma \in (1,3]$ and $q_2=\frac{\gamma+\vartheta}{1+\vartheta}>2$ when $\gamma>3$. 

Finally, we recall Murat's Lemma (see \cite{Mu,Mu',T,C,DC}) stating that 
\[
\{\text{Compact of } W_{loc}^{-1,q}(\Omega)\}\cap\{ \text{Bounded of } W_{loc}^{-1,r}(\Omega)\}\subseteq \{ \text{Compact of }W_{loc}^{-1,p}(\Omega) \},
\] 
for any $1<q\leq p<r\leq \infty$, which implies that
\[
\eta^\zeta(\rho^\varepsilon,\rho^\varepsilon u^\varepsilon)_t + q^\zeta(\rho^\varepsilon,\rho^\varepsilon u^\varepsilon)_x
\]
belong to a compact of $H_{loc}^{-1}(\Omega\times[0,\infty))$.
\end{proof}

Let us introduce the following notation. First, we omit the first two arguments ($\rho$ and $u$) in the entropy kernel (recall \eqref{entrkernel}) and denote
\[
\chi(\xi)=[\rho^{2\vartheta}-(u-\xi)^2]_+^{\Lambda}.
\]
Second, given any function $f(\rho,u)$ with growth slower than $\rho|u|^3+\rho^{\gamma+\max\{ 1,\vartheta\}}$, we denote
\[
f(\rho^\varepsilon,u^\varepsilon)\rightharpoonup \overline{f(\rho,u)}(x,t):=\langle \nu_{x,t},f(\rho,u)\rangle.
\]
In other words, the overline stands for integration with respect to the young measure. 

We can finally show that our parametrized Young measure $\nu_{x,t}$ satisfies the commutator relations. Remember that we are assuming that $\alpha=o(\varepsilon^{1/2})$, $\beta=o(\varepsilon)$ and $\delta=o(\varepsilon)$.

\begin{proposition}
For each test function $\zeta(s)\in \{ \pm 1,\pm s, s^2 \}$ we have that
\begin{equation}
\langle \nu_{t,x}, \eta^\zeta\rangle_t+\langle \nu_{x,t},q^\zeta\rangle_x \leq 0,\hspace{5mm}\langle \nu_{x,t},\eta^\zeta\rangle(0,\cdot)=\eta(\rho_0,\rho_0 u_0),\label{desigentr}
\end{equation}
in the sense of distributions. Moreover, $\nu_{x,t}$ satisfies the following commutator relation
\begin{equation}
\vartheta(\xi_2-\xi_1)(\hspace{.5mm}\overline{\chi(\xi_1)\chi(\xi_2)} - \overline{\chi(\xi_1)}\hspace{1.5mm}\overline{\chi(\xi_2)}\hspace{.5mm})=(1-\vartheta)(\hspace{.5mm}\overline{u\chi(\xi_2)}\hspace{1.5mm}\overline{\chi(\xi_1)}-\overline{u\chi(\xi_1)}\hspace{1.5mm}\overline{\chi(\xi_2)}\hspace{.5mm}),\label{commutator}
\end{equation}
where, as before, $\vartheta=(\gamma-1)/2$.
\end{proposition}

\begin{proof}
First we prove \eqref{desigentr}. Multiplying \eqref{E1rho} by $\eta_\rho^\zeta$ and \eqref{E1u} by $\eta_m^\zeta$ and adding the resulting equations we obtain
\begin{flalign}
&\eta^\zeta(\rho^\varepsilon,m^\varepsilon)_t + q^\zeta(\rho^\varepsilon,m^\varepsilon)_x & \nonumber\\
&\hspace{5mm}=(\varepsilon\eta_m^\zeta(\rho^\varepsilon,m^\varepsilon)u_x^\varepsilon - \delta\theta^\varepsilon p_\theta(\rho^\varepsilon))_x&\nonumber\\
&\hspace{15mm}-\big(\varepsilon u_x^\varepsilon-\delta\theta^\varepsilon p_\theta(\rho^\varepsilon)\big) \big( \eta_{mu}^\zeta(\rho^\varepsilon,m^\varepsilon)u_x^\varepsilon + \eta_{m\rho}^\zeta(\rho^\varepsilon,m^\varepsilon)\rho_x^\varepsilon\big)&\nonumber\\
&\hspace{25mm}-\eta_m^\zeta(\rho^\varepsilon,m^\varepsilon)\big( \frac{\beta}{2}|\mathbf{h}^\varepsilon|^2-\alpha g'(1/\rho^\varepsilon)h(|\psi^\varepsilon|^2)\big)_x.\label{entrmeasalm}
\end{flalign}

Let us show that all the terms on the RHS tend to zero as $\varepsilon\to 0$ in the sense of distributions, except possibly for the term $\varepsilon \eta_m^\zeta(\rho^\varepsilon,m^\varepsilon) |u_x^\varepsilon|^2$, which turns out to be nonpositive anyway.

From \eqref{entropexpl}, given any $\zeta\in C^2(\mathbb{R})$ we have
\[
\eta_{m}^\zeta(\rho,\rho u)=\int_{-1}^1\zeta'(u+\rho^\vartheta s)[1-s^2]_+^\Lambda ds.
\]

Hence
\[
\eta_{mu}^\zeta(\rho,\rho u)=\int_{-1}^1\zeta''(u+\rho^\vartheta s)[1-s^2]_+^\Lambda ds,
\]
and also
\[
\eta_{m\rho}^\zeta=(\rho,\rho u)=\vartheta\rho^{\vartheta-1}\int_{-1}^1\zeta''(u+\rho^\vartheta s)s[1-s^2]_+^\Lambda ds.
\]

Take $\zeta(s)\in \{ \pm 1, \pm s, s^2 \}$. Then, $|\eta_m^\zeta(\rho^\varepsilon,m^\varepsilon)|\leq C(1+|u^\varepsilon|)$, $\eta_{m\rho}^\zeta(\rho^\varepsilon,m^\varepsilon)=0$ and $0\leq\eta_{mu}^\zeta(\rho^\varepsilon,m^\varepsilon)\leq C$ (note that $\int s[1-s^2]_+^\Lambda ds =0$).

Let us recall \eqref{hL1infty} and \eqref{uL1infty}. Then,
\begin{align*}
&\int_0^T\int_\Omega \left| \eta_m^\zeta(\rho^\varepsilon,m^\varepsilon) \left(\frac{\beta}{2}|\mathbf{h}|^2\right)_x \right|dx\hspace{.5mm}ds\leq C\int_0^T\int_\Omega (1+|u^\varepsilon|)\beta|\mathbf{h}^\varepsilon\cdot\mathbf{h}_x^\varepsilon|dx\hspace{.5mm}ds\\
   &\hspace{10mm}\leq C\left( 1+\left( \int_0^T\int_\Omega |u_x^\varepsilon|^2 dx\hspace{.5mm}ds \right)^{1/2} \right)\left( \beta \int_0^T\int_\Omega |\mathbf{h}^\varepsilon\cdot\mathbf{h}_x^\varepsilon|^2 dx\hspace{.5mm}ds \right)^{1/2}\\
   &\hspace{10mm}\leq C\left( 1+\left( \int_0^T\int_\Omega |u_x^\varepsilon|^2 dx\hspace{.5mm}ds \right)^{1/2} \right)\left( \beta^{1/2} \int_0^T\int_\Omega |\mathbf{h}_x^\varepsilon|^2 dx\hspace{.5mm}ds \right)\\
   &\hspace{10mm}\leq C\frac{\beta^{1/2}}{\varepsilon^{1/2}}\left( \varepsilon^{1/2}+\left(\varepsilon \int_0^T\int_\Omega |u_x^\varepsilon|^2 dx\hspace{.5mm}ds \right)^{1/2}\right)\\
   &\hspace{10mm}\leq C\frac{\beta^{1/2}}{\varepsilon^{1/2}},
\end{align*}
which tends to zero as $\varepsilon\to 0$.

Similarly, recalling \eqref{gh}, that $\psi_x=\rho \psi_y$ and that the Jacobian of the coordinate change equals $\rho$, from Lemmas \ref{uniform1} and \ref{uniform2} we have
\begin{align*}
&\int_0^T\int_\Omega \left| \eta_m^\zeta(\rho^\varepsilon,m^\varepsilon) \left(\alpha g'(1/\rho^\varepsilon)h(|\psi^\varepsilon|^2)\right)_x \right|dx\hspace{.5mm}ds\\
&\hspace{10mm}\leq C\alpha\int_0^T\int_\Omega(1+|\rho^\varepsilon|^{1/2}|u^\varepsilon|)(|\rho^\varepsilon|^{\frac{\gamma-3}{2}}|\rho_x^\varepsilon|+|\psi_x^\varepsilon g'(1/\rho^\varepsilon)|)dx\hspace{.5mm}ds\\
&\hspace{10mm}\leq C\frac{\alpha}{\varepsilon^{1/2}},
\end{align*}
which also tends to zero as $\varepsilon\to 0$.

Next, using \eqref{EMtetint}
\begin{align*}
&\int_0^T\int_\Omega |\delta \theta^\varepsilon p_\theta(\rho^\varepsilon)(\eta_{mu}^\zeta(\rho^\varepsilon,m^\varepsilon)u_x^\varepsilon + \eta_{m\rho}^\zeta(\rho^\varepsilon,m^\varepsilon)\rho_x   )|dxds\\
&\hspace{10mm}\leq C\frac{\delta}{\varepsilon^{1/2}}\left(\int_0^T M_\theta(s)^2\int_\Omega \rho^\gamma dx ds\right)^{1/2} \left( \varepsilon\int_0^T\int_\Omega u_x^2 dxds \right)^{1/2}\\
&\hspace{10mm}\leq C\frac{\delta}{\varepsilon^{1/2}}
\end{align*}
which, tends to zero as well.

Finally, by the same token we have that
\[
\eta_m^\zeta(\rho^\varepsilon,m^\varepsilon)(\varepsilon u_x^\varepsilon - \delta \theta^\varepsilon p_\theta(\rho^\varepsilon))\to 0
\]
in the sense of distributions. Thus, taking $\varepsilon\to 0$ in \eqref{entrmeasalm} we obtain \eqref{desigentr}.

Let us now prove \eqref{commutator}. In view of Lemmas \ref{uniform3} and \ref{uniform4} and Proposition \ref{etaqcomp}, we can apply the celebrated Div-Curl Lemma (see \cite{Mu,T,T'}) in order to conclude that for any $C^2$ compactly supported functions $\zeta$ and $\phi$ we have
\[
\overline{\eta^\zeta q^\phi} - \overline{\eta^\phi q^\zeta} = \overline{\eta^\zeta}\hspace{1.5mm} \overline{q^\phi} - \overline{\eta^\phi}\hspace{1.5mm} \overline{q^\zeta}.
\]

Consequently,
\begin{flalign*}
&\int\zeta(\xi_1)\overline{\chi(\xi_1)}d\xi_1\int \phi(\xi_2)\overline{\vartheta\xi_2+(1-\vartheta)u)\chi(\xi_2)}\hspace{.5mm}d\xi_2 &\\
&\hspace{10mm}-\int\phi(\xi_2)\overline{\chi(\xi_2)}d\xi_2\int \zeta(\xi_1)\overline{\vartheta\xi_1+(1-\vartheta)u)\chi(\xi_1)}\hspace{.5mm}d\xi_1&\\
&\hspace{20mm}=\int \zeta(\xi_1)\phi(\xi_2) \overline{\chi(\xi_1)(\vartheta\xi_2+(1-\vartheta)u)\chi(\xi_2)}d\xi_1 d\xi_2 &\\
&\hspace{30mm}-\int \zeta(\xi_1)\phi(\xi_2) \overline{(\vartheta\xi_1+(1-\vartheta)u)\chi(\xi_1)\chi(\xi_2)}d\xi_1 d\xi_2. &
\end{flalign*}

As this holds for any $\zeta$ and $\phi$ we have
\begin{flalign*}
&\overline{\chi(\xi_1)}\hspace{1mm}\overline{\vartheta\xi_2+(1-\vartheta)u)\chi(\xi_2)} - \overline{\chi(\xi_2)}\hspace{1mm}\overline{\vartheta\xi_1+(1-\vartheta)u)\chi(\xi_1)}\\
&\hspace{70mm}=\vartheta (\xi_2-\xi_1)\overline{\chi(\xi_1)\chi(\xi_2)},&&
\end{flalign*}
which implies \eqref{commutator}.
\end{proof}

As aforementioned , with this proposition a hand, the arguments in \cite{CP,C,DC} to reduce the Young measures apply and we have shown the first part of Theorem \ref{vanish}.

\subsection{Limit for the transverse velocity, magnetic field and wave function}

Having strong convergence of the sequence $(\rho^\varepsilon,\rho^\varepsilon u^\varepsilon)$, the passage to the limit in equations \eqref{E1w} and \eqref{E1h} becomes a straightforward exercise. As pointed out before, the uniform estimates in Lemma \ref{uniform1} and the fact that we are leaving $\mu$ and $\nu$ fixed independently of $\varepsilon$, imply that $\beta \mathbf{h}^\varepsilon$ and $\beta \mathbf{w}^\varepsilon$ tend to zero in $L^2(\Omega\times(0,T))$ and $\beta u^\varepsilon\mathbf{h}^\varepsilon$ tends to zero in $L^1(\Omega\times(0,T))$. Accordingly, we have that $\mathbf{h}_{xx}^\varepsilon$ in the sense of distributions. Nonetheless, we also have a uniform bound for the $L^2(0,T;H_0^1(\Omega))$ for $\mathbf{h}^\varepsilon$ so that we can assume that it converges to some limit $\mathbf{h}$ weakly in $L^2(0,T;H_0^1(\Omega))$. This implies that, necessarily, $\mathbf{h}=0$ and the limit equation is satisfied trivially. Thus, for consistency, we demand that the initial data $\mathbf{h}_0^\varepsilon$ satisfies $\beta \mathbf{h}_0^\varepsilon\to 0$ in the sense of distributions and drop equation \eqref{E1hinfty}.

Moving on to equation \eqref{E1w}, the uniform estimates from Lemma \ref{uniform1} imply that $\mathbf{w}^\varepsilon$ has a subsequence (not relabelled) that converges to some limit $\mathbf{w}$ weakly in $L^2(0,T;H_0^1(\Omega))$. Since $(\rho^\varepsilon,\rho^\varepsilon u^\varepsilon)$ converges strongly to $(\rho,\rho u)$ we have that $\rho^\varepsilon \mathbf{w}^\varepsilon$ and $\rho^\varepsilon u^\varepsilon \mathbf{w}^\varepsilon$ converge to $\rho \mathbf{w}$ and $\rho u\mathbf{w}$, respectively, in the sense of distributions. As $\beta \mathbf{h}$ converges strongly to zero, we have that the limit functions $\rho$, $\rho u$ and $\mathbf{w}$ solve the limit equation \eqref{E1winfty}.

Regarding the initial data for the transverse velocity, we demand that $\rho_0^\varepsilon\mathbf{w}_0^\varepsilon$ converge to some limit $\rho_0 \mathbf{w_0}$ in the sense of distributions. Note that we specify the initial data for the limit equation in terms of the transverse momentum as vacuum is unavoidable in the limit functions. Accordingly, we have that the initial data is attained in the sense of distributions through the weak formulation of \eqref{E1winfty}:
\[
\int_0^t\int_\Omega \rho \mathbf{w}\varphi_tdx ds -\int_\Omega \rho_0\mathbf{w}_0\varphi|_{t=0}ds  -\int_0^t\int_\Omega\rho u\mathbf{w} \varphi_xdx ds =-\int_0^t\int_\Omega \mu \mathbf{w}_x\varphi_xdxds
\]
for any $\varphi\in C^\infty(\Omega\times([0,\infty))$.

Lastly, the passage to the limit in the nonlinear Schr\"{o}dinger equation is a direct consequence of Aubin-Lions Lemma, as explained in Section \ref{limit}. For consistency, we assume that the initial data $\psi_0^\varepsilon$ converges to $\psi_0$ in $H_0^1(\Omega_y)$, thereby concluding that $\psi^\varepsilon$ converges to the unique solution of the limit nonlinear Schr\"{o}dinger equation \eqref{E1Schinfty}.

We have thus proved the second and third parts of Theorem \ref{vanish}.

To conclude, we move on to discussing the limit passage in the thermal energy equation \eqref{E1Q}.

\subsection{Limit in the thermal energy equation}

As explained in Section \ref{limit} the limit process in the thermal energy equation \eqref{E1Q} is not straightforward on account of the nonlinearities. Also, the loss of regularity of the longitudinal velocity $u$ forces us to consider the inequality \eqref{E1Qineq} instead of \eqref{E1Qinfty}.

In order to justify the passage to the limit, we adapt some ideas in \cite{Fe}.

First, we observe that estimate \eqref{uniftetq} implies that $Q(\theta^\varepsilon)$ is uniformly bounded in $L^2(0,T;H^1(\Omega))$ and hence we can assume that 
\[
Q(\theta^\varepsilon)\rightharpoonup \overline{Q}\text{ weakly in }L^2(0,T;H^1(\Omega)).
\]

By the same token, $Q(\theta^\varepsilon)$ is uniformly bounded in $L^2(0,T;L^\infty(\Omega))$ which, in light of Lemma \ref{uniform1}, implies that $\rho^\varepsilon Q(\theta^\varepsilon)$ is uniformly bounded in the space $L^2(0,T;L^\gamma(\Omega))$.

Now, from equation \eqref{E1Q} and using Lemmas \ref{uniform1} and \ref{uniform5} we can easily see that $(\rho^\varepsilon Q(\theta^\varepsilon))_t$ is uniformly bounded in the space $L^1(0,T;H^{-3}(\Omega))$ (recall that $(\kappa \theta_x)_x$ can be written as $\mathcal{K}_{xx}$ and that by Lemma \ref{uniform5} $\mathcal{K}$ is uniformly bounded in $L^1(\Omega\times(0,T))$). With this, a direct application of Aubin-Lions lemma (see \cite{JLi,Au,Si}) and the fact that $\rho^\varepsilon$ converges strongly to $\rho$ imply that 
\[
\rho^\varepsilon Q(\theta^\varepsilon)\to\rho\hspace{.3mm} \overline{Q} \text{ strongly in } L^2(0,T;H^{-1}(\Omega)).
\]

We claim that
\begin{equation}
\int_0^T\int_\Omega \rho^\varepsilon Q(\theta^\varepsilon)^2\varphi dx ds\to \int_0^T\int_\Omega \rho\overline{Q}^2\varphi dx ds,\label{rhoQ2}
\end{equation}
as $\varepsilon\to 0$, for \textbf{any} $\varphi\in C_0^\infty(\Omega\times(0,T))$. Indeed, we have
\begin{align*}
&\int_0^T\int_\Omega (\rho^\varepsilon Q(\theta^\varepsilon)^2 -\rho\overline{Q}^2)\varphi dxds\\
&\hspace{2mm}\leq \int_0^T\int_\Omega (\rho^\varepsilon Q(\theta^\varepsilon) -\rho\overline{Q})(Q(\theta^\varepsilon)+\overline{Q})\varphi dxds + \int_0^T\int_\Omega Q(\theta^\varepsilon)\overline{Q}(\rho^\varepsilon-\rho)\varphi dxds.
\end{align*}

On the one hand we have
\begin{align*}
&\int_0^T\int_\Omega (\rho^\varepsilon Q(\theta^\varepsilon) -\rho\overline{Q})(Q(\theta^\varepsilon)+\overline{Q})\varphi dxds \\
&\leq \int_0^T \Vert (\rho^\varepsilon Q(\theta^\varepsilon)-\rho \overline{Q})(s)\Vert_{H^{-1}(\Omega)}\Vert (Q(\theta^\varepsilon)+\overline{Q})(s)\varphi\Vert_{H_0^1(\Omega)}ds\\
 &\leq  C_\varphi \int_0^T \Vert (\rho^\varepsilon Q(\theta^\varepsilon)-\rho \overline{Q})(s)\Vert_{H^{-1}(\Omega)}\big(\Vert Q(\theta^\varepsilon)(s)\Vert_{H^1(\Omega)}+\Vert \overline{Q}(s)\Vert_{H^1(\Omega)}\big)ds\\
  &\leq C_\varphi \Vert \rho^\varepsilon Q(\theta^\varepsilon)-\rho \overline{Q}\Vert_{L^2(0,T;H^{-1}(\Omega))}\times\\
  &\hspace{40mm}\times\big(\Vert Q(\theta^\varepsilon)(s)\Vert_{L^2(0,T;H^1(\Omega))}+\Vert \overline{Q}(s)\Vert_{L^2(0,T;H^1(\Omega))}\big),
\end{align*}
which tends to zero as $\varepsilon\to 0$.

On the other hand,
\[
\int_0^T\int_\Omega Q(\theta^\varepsilon)\overline{Q}(\rho^\varepsilon-\rho)\varphi dxds \to 0
\]
by the dominated convergence theorem (recall that $\rho^\varepsilon\to \rho$ a.e. in $\Omega\times (0,T)$), thus proving the claim.

Now, from \eqref{rhoQ2} we have that 
\begin{flalign*}
&\int_0^T\int_\Omega \rho Q(\theta^\varepsilon)\varphi dx ds=\int_0^T\int_\Omega (\rho-\rho^\varepsilon) Q(\theta^\varepsilon)\varphi dx ds + \int_0^T\int_\Omega \rho^\varepsilon Q(\theta^\varepsilon)\varphi dx ds\\
&\hspace{32mm}\to \int_0^T\int_\Omega \rho \overline{Q}\varphi dx ds.&&
\end{flalign*}
also by the dominated convergence theorem.

This can be interpreted as convergence of norms in a weighted $L_{\rho \varphi}^2$ space. In particular, we have
\begin{equation}
Q(\theta^\varepsilon)\to \overline{Q} \text{ a.e. in }\{ (x,t)\in \Omega\times(0,T):\rho(x,t)>0 \}.\label{convQ}
\end{equation}

Since $Q$ is strictly increasing (recall our hypotheses \eqref{Q}) we can define $\overline{\theta}:=Q^{-1}(\hspace{.5mm}\overline{Q}\hspace{.5mm})$ and we have that
\begin{align*}
0&=\lim_{\varepsilon\to 0}\int_0^T\int_\Omega (Q(\theta^\varepsilon)-\overline{Q})(\theta^\varepsilon-\overline{\theta})\mathbbm{1}_{\{\rho>0\}} dxds\\
 &\geq C^{-1}\int_0^T\int_\Omega(\theta^\varepsilon-\overline{\theta})^2\mathbbm{1}_{\{\rho>0\}}dxds,
\end{align*}
and hence,
\[
\theta^\varepsilon\to\overline{\theta} \text{ in }L^2(\{ \rho>0 \}).
\]

This last bit of information guarantees that we can pass to the limit in the first two terms of equation \eqref{E1Q} (remember that $\rho^\varepsilon u^\varepsilon\to \rho u$ strongly). Regarding the third term in that equation, we are assuming that there is a coefficient $\delta=o(\varepsilon)$ multiplying it, and by the estimates in Lemma \ref{uniform1} it converges to zero in the sense of distributions. 

All there is left to do, then, is justify the passage to the limit in the second order term on the right hand side. For this we need the following lemma (see \cite[Proposition 2.1]{Fe}).
\begin{lemma}\label{lemmaL1weak}
Let $O\subseteq \mathbb{R}^M$ be a bounded open set. Let $\{v_n\}_{n=1}^\infty$ be a sequence of measurable functions,
\[
v_n:O\to \mathbb{R}^N,
\]
such that 
\[
\sup_{n\geq 1}\int_O \Phi(|v_n|)d\xi < \infty
\]
for a certain continuous function $\Phi:[0,\infty)\to[0,\infty)$.

Then, there exists a subsequence (not relabelled) such that 
\[
\zeta(v_n)\to \overline{\zeta(v)} \text{ weakly in }L^1(O)
\]
for all continuous functions $\zeta:\mathbb{R}^N\to\mathbb{R}$ satisfying
\[
\lim_{|\mathbf{z}|\to\infty}\frac{\zeta(\mathbf{z})}{\Phi(\mathbf{z})}=0.
\]
\end{lemma}

Fix $0<\omega <1$ and choose $\zeta:[0,\infty)\to [0,\infty)$ as $\zeta(z)=\frac{1}{(1+z)^\omega}$. Then, multiplying \eqref{E1Q} by $\zeta(\theta^\varepsilon)$ and using equation \eqref{E1rho} we have
\begin{align}
&(\rho^\varepsilon Q_\zeta(\theta^\varepsilon))_t + (\rho^\varepsilon u^\varepsilon Q_\zeta(\theta^\varepsilon))_x +\ \delta \theta^\varepsilon p_\theta(\rho^\varepsilon)\zeta(\theta^\varepsilon)u_x^\varepsilon \nonumber\\
&\hspace{15mm}= (\mathcal{K}_\zeta(\theta^\varepsilon))_{xx} +\frac{\omega\kappa(\theta^\varepsilon)|\theta_x^\varepsilon|^2}{(1+\theta^\varepsilon)^\omega} + \frac{\varepsilon |u_x^\varepsilon|^2+\mu|\mathbf{w}_x^\varepsilon|^2+\nu |\mathbf{h}_x^\varepsilon|^2}{(1+\theta^\varepsilon)^\omega},\label{E1Qzet}
\end{align}
where, $Q_\zeta$ and $\mathcal{K}_\zeta$ are given by
\[
Q_\zeta (\theta):=\int_0^\theta \frac{C_\vartheta(z)}{(1+z)^\omega}dz,\hspace{10mm}\mathcal{K}_\zeta(\theta):=\int_0^\theta \frac{\kappa(z)}{(1+z)^\omega}dz
\]

From the strong convergence of $\rho^\varepsilon$, $\rho^\varepsilon u^\varepsilon$, the strong convergence of $\theta^\varepsilon$ in $\{\rho>0\}$ and the uniform estimates we see that
\[
\begin{rcases}
\rho^\varepsilon Q_\zeta(\theta^\varepsilon)\to \rho Q_\zeta(\overline{\theta})\\
\rho^\varepsilon u^\varepsilon Q_\zeta(\theta^\varepsilon)\to\rho u Q_\zeta(\overline{\theta})\\
\delta \theta^\varepsilon p_\theta(\rho^\varepsilon)\zeta(\theta^\varepsilon)u_x^\varepsilon \to 0
\end{rcases}\text{ weakly in }L^1(\Omega\times (0,T)).
\]

Next, using Lemma \ref{lemmaL1weak} we see that 
\[
\mathcal{K}_\zeta(\theta^\varepsilon)\to \overline{\mathcal{K}}_\zeta \text{ weakly in }L^1(\Omega\times(0,T)),
\]
for some $\overline{\mathcal{K}}_\zeta$ that satisfies
\[
\rho \overline{\mathcal{K}}_\zeta=\rho  \mathcal{K}_\zeta(\overline{\theta}), \text{ on }\Omega\times(0,T).
\]

Now, let $\varphi$ be a test function such that
\begin{equation}
\varphi\geq 0,\hspace{5mm} \varphi\in W^{2,\infty}(\Omega\times(0,T)),\hspace{5mm} \varphi_x|_{\partial \Omega}=0,\hspace{5mm} \text{supp}\varphi\subseteq \overline{\Omega}\times[0,T).\label{test}
\end{equation}

For any such test function we have that
\[
\int_0^T\int_\Omega\frac{\mu|\mathbf{w}_x|^2}{(1+\overline{\theta})^\omega}\varphi dxds\leq\liminf_{\varepsilon\to 0}\int_0^T\int_\Omega\frac{\mu|\mathbf{w}_x^\varepsilon|^2}{(1+\theta^\varepsilon)^\omega}\varphi dxds.
\]

Thus, multiplying \eqref{E1Qzet} by $\varphi$, integrating and taking the limit as $\varepsilon\to 0$ we obtain
\begin{flalign}
&\int_0^T \int_\Omega \big(\rho Q_\zeta(\overline{\theta})\varphi_t + \rho u Q_\zeta(\overline{\theta})\varphi_x + \overline{\mathcal{K}}_\zeta\varphi_{xx}\big) dx ds &\nonumber\\
&\hspace{35mm}\leq -\int_0^T\int_\Omega \frac{\mu|\mathbf{w}_x|^2}{(1+\overline{\theta})^\omega}\varphi dxds -\int_\Omega \rho_0 Q_\zeta(\theta_0)\varphi|_{t=0}dx.&\label{Qzetfrac}
\end{flalign}
For this last term we are assuming that $\rho_0^\varepsilon Q(\theta_0^\varepsilon)\to\rho_0 Q(\theta_0)$.

Now, note that 
\[
\frac{1}{(1+z)^\omega}\nearrow 1, \text{ as }\omega\to 0,
\]
then, using the monotone convergence theorem we see that 
\[
\overline{\mathcal{K}}_\zeta\nearrow\overline{K},
\]
where,
\[
\rho\overline{\mathcal{K}}=\rho \mathcal{K}(\hspace{.3mm}\overline{\theta}\hspace{.3mm}),
\]
and
\[
\int_0^T\int_\Omega \overline{\mathcal{K}}dxds \leq \liminf_{\varepsilon\to 0}\int_0^T\int_\Omega \mathcal{K}(\theta^\varepsilon)dxds.
\]

Finally, we can define $\theta:=\mathcal{K}^{-1}(\overline{\mathcal{K}})$ and take the limit as $\omega\to 0$ in \eqref{Qzetfrac} in order to conclude that the nonnegative function $\theta$ satisfies
\begin{flalign}
&\int_0^T \int_\Omega \big(\rho Q(\theta)\varphi_t + \rho u Q(\theta)\varphi_x + \mathcal{K}(\theta)\varphi_{xx}\big) dx ds &\nonumber\\
&\hspace{40mm}\leq -\int_0^T\int_\Omega \mu|\mathbf{w}_x|^2\varphi dxds -\int_\Omega \rho_0 Q(\theta_0)\varphi|_{t=0}dx,&\label{E1Qinftyineq}
\end{flalign}
for any test function that satisfies \eqref{test}; which is the weak formulation of inequality \eqref{E1Qineq}.

Finally, let us show that \eqref{E1Einftyweak} holds. From the energy identity \eqref{difE1} we have
\begin{flalign*}
&\int_\Omega \left( \rho^\varepsilon\left(e(\rho^\varepsilon,\theta^\varepsilon) +\tfrac{1}{2}|u^\varepsilon|^2 + \tfrac{1}{2}|\mathbf{w}^\varepsilon|^2  \right)+\frac{\beta}{2}|\mathbf{h}^\varepsilon|^2 \right) dx &\\
&\hspace{30mm}+\int_{\Omega_y}\Big(\alpha g(v^\varepsilon) h(|\psi^\varepsilon|^2) + \tfrac{1}{2} |\psi_y^\varepsilon|^2 + \tfrac{1}{4}|\psi^\varepsilon|^4\Big)  dy&\\
&= \int_\Omega \left( \rho_0^\varepsilon\left(e(\rho_0^\varepsilon,\theta_0^\varepsilon) +\tfrac{1}{2}|u_0^\varepsilon|^2 + \tfrac{1}{2}|\mathbf{w}_0^\varepsilon|^2  \right)+\frac{\beta}{2}|\mathbf{h}_0^\varepsilon|^2 \right) dx &\\
&\hspace{30mm}+\int_{\Omega_y}\Big(\alpha g(v_0^\varepsilon) h(|\psi_0^\varepsilon|^2) + \tfrac{1}{2} |\psi_{0y}^\varepsilon|^2 + \tfrac{1}{4}|\psi_0^\varepsilon|^4\Big)dy.&
\end{flalign*}

By assumption, the right hand side tends to 
\begin{flalign*}
& \int_\Omega \rho_0\left(e(\rho_0,\theta_0) +\tfrac{1}{2}|u_0|^2 + \tfrac{1}{2}|\mathbf{w}_0|^2  \right)  dx +\int_{\Omega_y}\Big(\tfrac{1}{2} |\psi_{0y}|^2 + \tfrac{1}{4}|\psi_0|^4\Big)dy.&
\end{flalign*}

By lower semi-continuity we have that 
\[
\int_\Omega \left( \tfrac{1}{2}\rho |\mathbf{w}|^2  \right)dx\leq\int_\Omega \left( \tfrac{1}{2}\rho^\varepsilon |\mathbf{w}^\varepsilon|^2  \right)dx
\]
and by strong convergence in $\Omega\times(0,T)$ we have that 
\begin{flalign*}
& \int_\Omega \rho\left(e(\rho,\theta) +\tfrac{1}{2}|u|^2 + \tfrac{1}{2}|\mathbf{w}|^2  \right)(t)  dx +\int_{\Omega_y}\Big(\tfrac{1}{2} |\psi_{y}|^2 + \tfrac{1}{4}|\psi|^4\Big)(t)dy.&\\
&\hspace{10mm}\leq \int_\Omega \rho_0\left(e(\rho_0,\theta_0) +\tfrac{1}{2}|u_0|^2 + \tfrac{1}{2}|\mathbf{w}_0|^2  \right)  dx +\int_{\Omega_y}\Big(\tfrac{1}{2} |\psi_{0y}|^2 + \tfrac{1}{4}|\psi_0|^4\Big)dy.&
\end{flalign*}
for a.e. $t\in(0,T)$.

Finally, we recall that the unique solution of the nonlinear Schr\"{o}dinger equation has conservation of energy:
\[
\int_{\Omega_y}\Big(\tfrac{1}{2} |\psi_{y}|^2 + \tfrac{1}{4}|\psi|^4\Big)(t)dy=\int_{\Omega_y}\Big(\tfrac{1}{2} |\psi_{0y}|^2 + \tfrac{1}{4}|\psi_0|^4\Big)dy
\]
which implies \eqref{E1Einftyweak}, thus concluding the proof of Theorem \ref{vanish}.

\section*{Acknowledgements}
D.R.~Marroquin thankfully acknowledges the support from CNPq, through grants proc. 150118/2018-0 and poc. 140375/2014-7.

\section*{References}


\end{document}